\documentclass[11pt,a4paper,reqno]{amsart}

\usepackage{mathrsfs} 

\usepackage{amsmath,amssymb,graphics,epsfig,color,enumerate,psfrag}
\usepackage{dsfont}  %
\usepackage{verbatim}
\usepackage[normalem]{ulem} 
\usepackage[colorlinks]{hyperref}
\hypersetup{ hidelinks = true }

\newtheorem{Theorem}{Theorem}[section]

\newtheorem{Lemma}{Lemma}[section]
\newtheorem{Proposition}[Theorem]{Proposition}

\theoremstyle{definition}
\newtheorem{Remark}{Remark}[section]

\newtheorem{Definition}{Definition}[section]

\newcommand{\cA}{\ensuremath{\mathcal A}}
\newcommand{\cB}{\ensuremath{\mathcal B}}

\newcommand{\cE}{\ensuremath{\mathcal E}}
\newcommand{\cF}{\ensuremath{\mathcal F}}

\newcommand{\cL}{\ensuremath{\mathcal L}}

\newcommand{\cN}{\ensuremath{\mathcal N}}
\newcommand{\cO}{\ensuremath{\mathcal O}}

\newcommand{\cS}{\ensuremath{\mathcal S}}

\newcommand{\bbE}{{\ensuremath{\mathbb E}} }

\newcommand{\bbN}{{\ensuremath{\mathbb N}} }

\newcommand{\bbP}{{\ensuremath{\mathbb P}} }

\newcommand{\bbR}{{\ensuremath{\mathbb R}} }

\newcommand{\bbT}{{\ensuremath{\mathbb T}} }

\newcommand{\bbZ}{{\ensuremath{\mathbb Z}} }

\let\a=\alpha \let\b=\beta   \let\d=\delta  \let\e=\varepsilon
 \let\g=\gamma       \let\l=\lambda
\let\m=\mu      \let\o=\omega      
  \let\s=\sigma \let\t=\tau   
  \let\z=\zeta
\let\D=\Delta

\newcommand{\dd}{\mathrm{d}}

\newcommand{\one}{\mathds{1}}
\newcommand{\modN}{\;(\mathrm{mod}\, N)}

\author[L.\ Levine]{Lionel Levine}
\address{Lionel Levine.
Department of Mathematics, Cornell University, Ithaca, NY 14853.
}
\email{levine@math.cornell.edu}
\thanks{LL was supported by \href{http://www.nsf.gov/awardsearch/showAward?AWD_ID=1455272}{NSF DMS-1455272} and a Sloan Fellowship.}

\author[V. Silvestri]{Vittoria Silvestri}
\address{Vittoria Silvestri. 
Statslab, Wilberforce road, Cambridge, CB3 0WA, United Kingdom.}
\email{V.Silvestri@maths.cam.ac.uk}

\title{How long does it take for Internal DLA to forget its initial profile?}

\begin{document}
\maketitle
\begin{center}
\today
\end{center}

\begin{abstract}
Internal DLA is a discrete model of a moving interface. On the cylinder graph $\bbZ_N \times \bbZ$, a particle starts uniformly on $\bbZ_N \times \{0\}$ and performs simple random walk on the cylinder until reaching an unoccupied site in $\bbZ_N \times \bbZ_{\geq 0}$, which it occupies forever. This operation defines a Markov chain on subsets of the cylinder.
We first show that a typical subset is rectangular with at most logarithmic fluctuations. We use this to prove that two Internal DLA chains started from different typical subsets can be coupled with high probability by adding order $N^2 \log N$ particles. For a lower bound, we show that at least order $N^2$ particles are required to forget which of two independent typical subsets the process started from.
\end{abstract}

\tableofcontents


\section{Introduction}
Internal Diffusion-Limited Aggregation (IDLA) models a random subset of $\bbZ^d$ that grows in proportion to the harmonic measure on its boundary seen from an internal point. 
More precisely, let $A(0)=\{ 0\}$ denote the subset (``cluster'') at time $t=0$. For integer times $t\geq 1$ inductively define
	\begin{equation} \label{eq:newparticle} A(t) = A(t-1) \cup \{ Z_t \} , \end{equation}
where $Z_t$ denotes the exit location from $A(t-1)$ of a simple random walk on $\bbZ^d$ starting from $0$, independent of the past. The process $(A(t))_{t\geq 0}$ is a Markov chain on the space of connected subsets of $\bbZ^d$. As $t\to \infty $, the asymptotic shape of $A(t)$ is an origin-centered Euclidean ball \cite{lawler1992internal} and its fluctuations from the ball are at most logarithmic in dimensions $d \geq 2$ \cite{asselah2013logarithmic,asselah2013sublogarithmic,jerison2012logarithmic,jerison2013internal}; see also \cite{asselah2014lower} for a lower bound on the fluctuations. 
Space-time averages of the fluctuations converge to a logarithmically correlated Gaussian field \cite{jerison2014internal,jerison2014internal2}. 

In the setting of the cylinder graph \cite{jerison2014internal2} it was asked what happens if the process is initiated with a cluster $A(0)$ other than a singleton: How long does it take for IDLA to forget the shape of $A(0)$? 
Our main results, Theorems \ref{th:ZNmain_teo} and \ref{th:ZNlower_bound} below, give upper and lower bounds that match up to a log factor.

Fix an integer $N\geq 3$, and let $\bbZ_N \times \bbZ$ denote the cylinder graph with the cycle on $N$ vertices as base graph. We refer to $x \in \bbZ_N$ as the horizontal coordinate, and to $y \in \bbZ$ as the vertical coordinate. For $k\in \bbZ$, we call $\{ y=k\} := \bbZ_N \times \{k\}$ the $k^{th}$ level of the cylinder, and $R_k := \{ (x,y) : y\leq k\}$ the infinite rectangle of height $k$.
 
 
It is sometimes convenient to formulate the growth in terms of aggregation of particles. Let $A(0)$ be the union of the lower half-cylinder $R_0$ with a finite (possibly empty) set of sites in the upper half-cylinder. At time zero, each site in $A(0)$ is occupied by a particle.
At each discrete time step, a new particle is released from a uniform point on level $0$, and performs a simple random walk until reaches an unoccupied site, which it occupies forever.
Motion of particles is instantaneous: we do not take into account how many random walk steps are required for a particle reach an unoccupied site, but rather increment the time by $1$ after it does so.
Formally, we define $A(t)$ inductively according to \eqref{eq:newparticle}, where $Z_t$ is the the exit location from $A(t-1)$ of a simple random walk on the cylinder graph starting from a uniform random site on level $0$, independent of the past. Note that this is equivalent to adding a new site to the cluster according to the harmonic measure on its exterior boundary seen from level $-\infty$. 
When $A(0) = R_0$ we say that the process is ``starting from flat''.

In the cylinder setting there are two parameters, the size $N$ of the base graph, and the time $t$. 
Just as large IDLA clusters on $\bbZ^d$ are logarithmically close to Euclidean balls, it is natural to expect large
 IDLA clusters on the cylinder to be logarithmically close to filled rectangles. 
 When $t = N^2$ this was stated in \cite{jerison2014internal2} (but not proved there, as the proof method is the same as in \cite{jerison2012logarithmic}).
Our first result extends this result to large times $t \leq N^m$, which we will later use to control the fluctuations of stationary clusters. 
\begin{Theorem} \label{pr:height_bound}
Let $(A(t))_{t\geq 0}$ be an IDLA process on $\bbZ_N \times \bbZ$ starting from the flat configuration $A(0)= R_0$. 
For any $\g >0$, $m\in \bbN$ there exists a constant  $b_{\g ,m}$, depending only on  $\g  ,m$, such that 
	\begin{equation} \label{eq:height_bound}
	 \bbP \Big(  R_{\frac{t}{N} - b_{\g ,m}  \log N }
	\subseteq  A(t) \subseteq
	R_{\frac{t}{N} + b_{\g ,m }  \log N  } \;\; \forall t\leq N^m \Big) \geq 1-N^{-\g}  
	\end{equation}
for $N$ large enough. 
\end{Theorem}
%

We prove the above result in two steps. To start with, we argue that \eqref{eq:height_bound} holds for $T \leq (N \log N )^2$, which can be shown by adapting the Jerison, Levine and Sheffield \linebreak arguments~\cite{jerison2012logarithmic} to the cylinder setting (cf.\ Theorem \ref{th:JLS}). We then invoke the Abelian\linebreak property of IDLA (cf.\ Section \ref{sec:Abelian}) to build large clusters by piling up nearly rectangular blocks of $\cO ( N^2 \log N)$ particles each.


Suppose now that the IDLA process on $\bbZ_N \times \bbZ$ is not initiated from the flat configuration $R_0$, but rather from an arbitrary connected cluster $A(0) \supset R_0$. How long does it take for the process $A(t)$ to forget that it did not start from flat? 
Clearly, the answer to this question very much depends on $A(0)$. For example, it will take an arbitrarily large time to forget an arbitrarily tall initial profile. On the other hand, most profiles are unlikely for IDLA dynamics. This leads us to the following related question: \vspace{1mm}

\hspace{1.3cm} \emph{How long does it take for IDLA to forget a \underline{typical} initial profile?} \\
\vspace{1mm}
To define ``typical'' let 
	\[ \Omega := \{ A \subseteq \bbZ_N \times \bbZ \,:\, A = R_0 \cup F \text{ for some finite $F$} \} \]
denote the set of clusters which are completely filled up to level $0$. This is 
 the state space of IDLA. On $\Omega $ we introduce the following 
 shift procedure: each time the cluster is completely filled up to level $k>0$, we shift the cluster down by $k$. 
\begin{Definition}[Shifted IDLA]\label{def:shifted}
Let $\cS$ be the map from the space of IDLA configurations to itself defined as follows. For a given cluster $A$, let
	\[ k_A := \max \{ k\geq 0 : 
	R_k \subseteq A \}  \]
be the height of the maximal filled (infinite) rectangle contained in $A$.
The \emph{downshift} of $A$ is the cluster 
	\[ \cS (A) := \{ (x,y-k_A) : (x,y) \in A \}, \]
Note that $k_{\cS (A)} = 0$, so $\cS (\cS (A)) = \cS(A)$.   
Now from the IDLA chain  $(A(t))_{t\geq 0}$ we set
	\[ A^*(t) := \cS ( A(t)) \]
for all $t\geq 0$. We will refer to $(A^*(t))_{t\geq 0}$ as the shifted process associated to $(A(t))_{t\geq 0}$. 
\end{Definition}

The shifted process defines a new Markov chain on the same configuration space $\Omega$. While the original IDLA chain is transient, we will see that Shifted IDLA is a positive recurrent Markov chain (cf.\ Remark \ref{remark_pr}) and it thus has a stationary distribution, which we denote by $\mu_N$. 

\begin{figure}[!h] \label{fig:shift}
  \centering
    \includegraphics[width=0.6\textwidth]{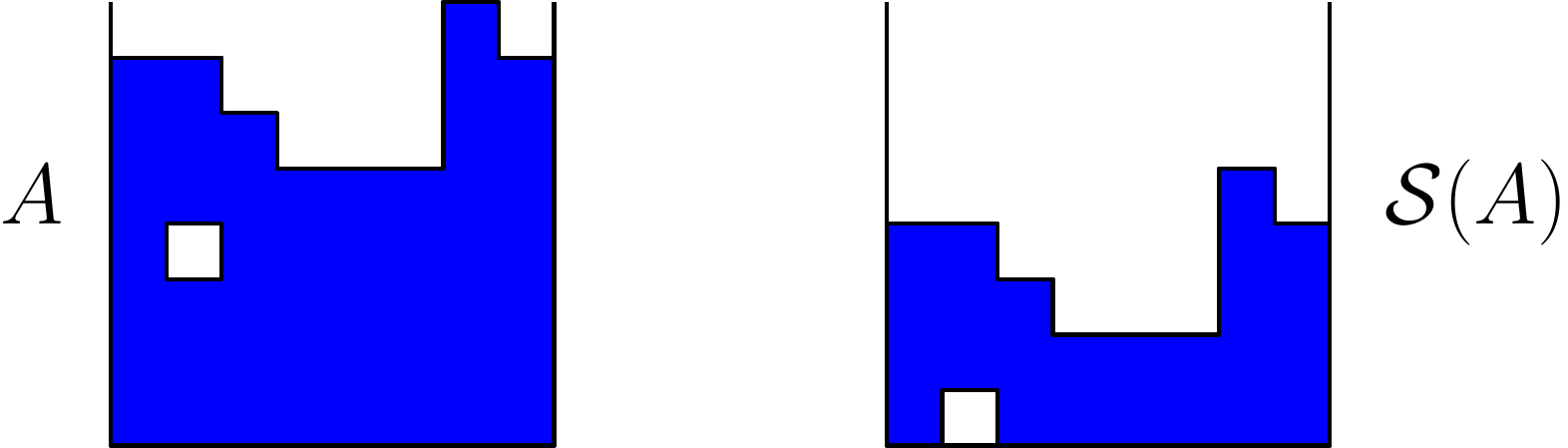}
    \caption{An IDLA cluster $A$ and its shifted version $\cS (A)$.}
\end{figure}

Recall from \cite{aldous1998microsurveys} that a probability distribution 
is called a \emph{warm start} for a given Markov chain if it is upper bounded by a constant factor times the  
stationary distribution of the chain.  
We relax this definition below. 
\begin{Definition}[Lukewarm start]
For $k\in \bbN$, a probability measure $\nu_N$ on $\Omega$ is said to be a \emph{$k$-lukewarm start} (for Shifted IDLA) if 
	\[ \nu_N(A) \leq N^k \mu_N(A)  \qquad \forall A \in \Omega . \]
We say that $\nu_N$ is a lukewarm start if it is a $k$-lukewarm start for some $k\in \bbN$.
\end{Definition}
\begin{Definition}[Typical clusters]
A random cluster $A \in \Omega$ is said to be a \emph{typical cluster} if it is distributed according to some lukewarm start $ \nu_N$ for Shifted IDLA. 
 If $\nu_N = \mu_N$ then $A$ is said to be a \emph{stationary cluster}. 
\end{Definition}
For $A \in \Omega$, let $|A|$ denote the number of occupied sites above level $0$, that is 
	\[ |A| := \sharp \{ (x,y) \in A : y>0 \}  . \]
Let further $h(A)$ denote the height of $A$, that is 
	\[ h(A) = \max\{ y : (x,y) \in A \mbox{ for some }x\in \bbZ_N\}, \]
so that $h(A) \geq 0$ for all $A \in \Omega$. 

Our next result is a bound for the height of a typical cluster: it is at most logarithmic in $N$ with high probability. 

\begin{Theorem}[Height of typical clusters]\label{pr:typical_height}
For any $\g >0$, $k\in \bbN$ there exists a constant $ c_{\g ,k} $, depending only on $\g $ and $k$, such that for all $k$-lukewarm starts $\nu_N$ it holds 
	\begin{equation} \label{numu}
	\nu_N \Big( \big\{ A : h(A) > c_{\g ,k}  \log N  \big\} \Big) 
	\leq  N^{-\g}  .
	\end{equation}
for $N$ large enough. 
\end{Theorem}
In particular, taking $k=0$ we see that stationary IDLA clusters have logarithmic height with high probability. This gives nontrivial information on the stationary measure of a nonreversible Markov chain on an infinite state space. 

\begin{Remark}
We believe that this bound is tight, in the sense that for $A_N \sim \mu_N$ the ratio $h(A_N )/\log N$ converges in probability to a positive constant as $N \to \infty$. See \cite[Figure~9]{friedrich2013fast} for some numerical evidence.
\end{Remark}

We use Theorem \ref{pr:typical_height} to bound from above the time it takes for IDLA to forget a typical initial profile. 
\begin{Theorem}[The upper bound] \label{th:ZNmain_teo}
For any $\g >0$, $k\in \bbN$ there exist a constant $d_{\g , k}$ and a set $\Omega_{\g ,k} \subseteq \Omega$, depending only on $\g$ and $k$, such that the following hold for all sufficiently large $N$. 
For any $k$-lukewarm start $\nu_N$ on $\bbZ_N \times \bbZ$, we have	
	\[ \nu_N (\Omega_{\g , k} ) \geq 1-N^{-\g}. \]
Moreover, for any $A_0, A'_0 \in \Omega_{\g ,k}$ with $|A_0| = |A'_0|$, writing $t_{\g ,k} = d_{\g , k} N^2 \log N$ for brevity, we have	
\[  \| P(t_{\g ,k} ) - P'(t_{\g , k} ) \|_{TV} \leq N^{-\g} \]
where $(A(t))_{t\geq 0}$ and $(A'(t))_{t\geq 0}$ are IDLA processes starting from $A_0$ and $A'_0$ respectively, and $P(t)$ and $P'(t)$ denote the laws of $A(t)$ and $A'(t)$ respectively.
\end{Theorem}
Thus IDLA forgets a typical initial state, and hence in particular a stationary initial state, in $\cO (N^2 \log N)$ steps with high probability.

\begin{Remark}
The assumption $|A_0|=|A'_0|$ could be relaxed to $|A_0| \equiv |A'_0| \modN $, by shifting the smaller cluster to include some full rows of $N$ sites each.
To remove the cardinality assumption entirely, one could consider ``lazy'' IDLA processes: at each discrete time step, a particle is added with probability $1/2$, and otherwise nothing happens.  Then the difference $|A_t| - |A'_t| \modN$ is a lazy simple random walk on an $N$-cycle, so with probability at least $1-N^{-\g}$ there is a time $T \leq d_\g N^2 \log N$ such that $|A_T| \equiv |A'_T| \modN $.
\end{Remark}

When the initial profiles are stationary, we can complement the upper bound in Theorem \ref{th:ZNmain_teo} with the following lower bound.

\begin{Theorem}[The lower bound] \label{th:ZNlower_bound}
For any $\d , \e >0$ there exist disjoint subsets $\Omega_\d ,\Omega '_\d$ of~$\Omega$ such that 
	\begin{equation}\label{omegas}
    \mu_N ( \Omega_\d ) \geq \frac{1}{2} - \d  \, , \qquad 
    \mu_N ( \Omega '_\d ) \geq \frac{1}{2} - \d
	\end{equation}
and a constant $\a = \a (\d , \e )>0$ such that the following holds. 
Let $(A(t))_{t\geq 0}$ and $(A'(t))_{t\geq 0}$ be two IDLA processes on $\bbZ_N \times \bbZ$ starting from $A_0 \in \Omega_\d$, $A'_0 \in \Omega '_\d $, and denote by  $P(t) , P'(t)$ the laws of  $A(t)$, $A'(t)$ respectively. Then 
	\begin{equation} \label{TVlow}
	 \| P(\a N^2 ) - P'(\a N^2 ) \|_{TV} > 1- \e
	 \end{equation}
for $N$ large enough. 
\end{Theorem}
The above theorem tells us that two independently sampled stationary profiles $A,A'$ are, with probability arbitrarily close to $1/2$, different enough for IDLA to need order~$N^2$ steps to forget from which one it started. To prove this, we identify a slow-mixing statistic based on the second eigenvector of simple random walk on $\bbZ_N$ (see Definition \ref{def:imbalance} in Section~\ref{sec:observable}).


\subsection{Outline of the proofs} \label{sec:sketch}
We start by showing that IDLA forgets polynomially high profiles in polynomial time, as stated in the following theorem.
\begin{Theorem}\label{th:main2}
Let $A_0 , A'_0$ be any two clusters in $\Omega$ with $|A_0|=|A'_0| $ and define 
	\[ h_0 = \max\{ h(A_0) , h(A'_0) \}.\] 
Assume that $h_0 \leq N^m$ for some fixed $m \in \bbN$. 
Let $(A(t))_{t\geq 0}$ and $(A'(t))_{t\geq 0}$ be two IDLA processes starting from $A_0 , A'_0$ and denote the laws of $A(t)$, $A'(t)$ by $P(t) , P'(t)$ respectively. 
Then for any $\g >0$ there exists a constant $d'_{\g , m} $, depending only on $\g $ and $m$, such that for 
	\[ t_{\g , m} = h_0 N + d'_{\g , m } N^2 \log N\] 
and $N$ large enough, it holds
	\[  \| P(t_{\g ,m} ) - P'(t_{\g , m} ) \|_{TV} \leq N^{-\g} . \]
\end{Theorem}

Theorem \ref{th:main2} is proved via a coupling argument that we spell out below. 
We then combine it with Theorem \ref{pr:height_bound} to show that stationary clusters have at most logarithmic height with high probability (cf.\ Theorem~\ref{pr:typical_height}). The upper bound (cf.\ Theorem~\ref{th:ZNmain_teo}), is a simple corollary of these two results. 
We sketch here the main ideas behind these proofs. \vspace{1.5mm}

\subsubsection{Theorem \ref{th:main2}: the water level coupling.}
Let $A_0$ and $A'_0$ be two clusters in $\Omega$ with $n_0:=|A_0|=|A'_0|$ and  $h_0 = \max\{ h(A_0) , h(A'_0) \} \leq N^m$.
We are going to build large IDLA clusters starting from $A_0 , A'_0$ in a convenient way.
Let $t_{\g , m} = h_0 N + d'_{\g , m} N^2 \log N$ as in Theorem \ref{th:main2}. 
Introduce a new IDLA cluster $W_0$, independent of everything else, built by adding $t_{\g , m}$ particles to the flat configuration~$R_0$. Note that since $W_0$ contains polynomially many particles, by Theorem \ref{pr:height_bound} it will be completely filled up to height $h_0+b_{\g , m} N\log N$ for a suitable constant $b_{\g , m }$ with high probability.

\begin{figure}[h!]
  \centering
    \includegraphics[width=.72\textwidth]{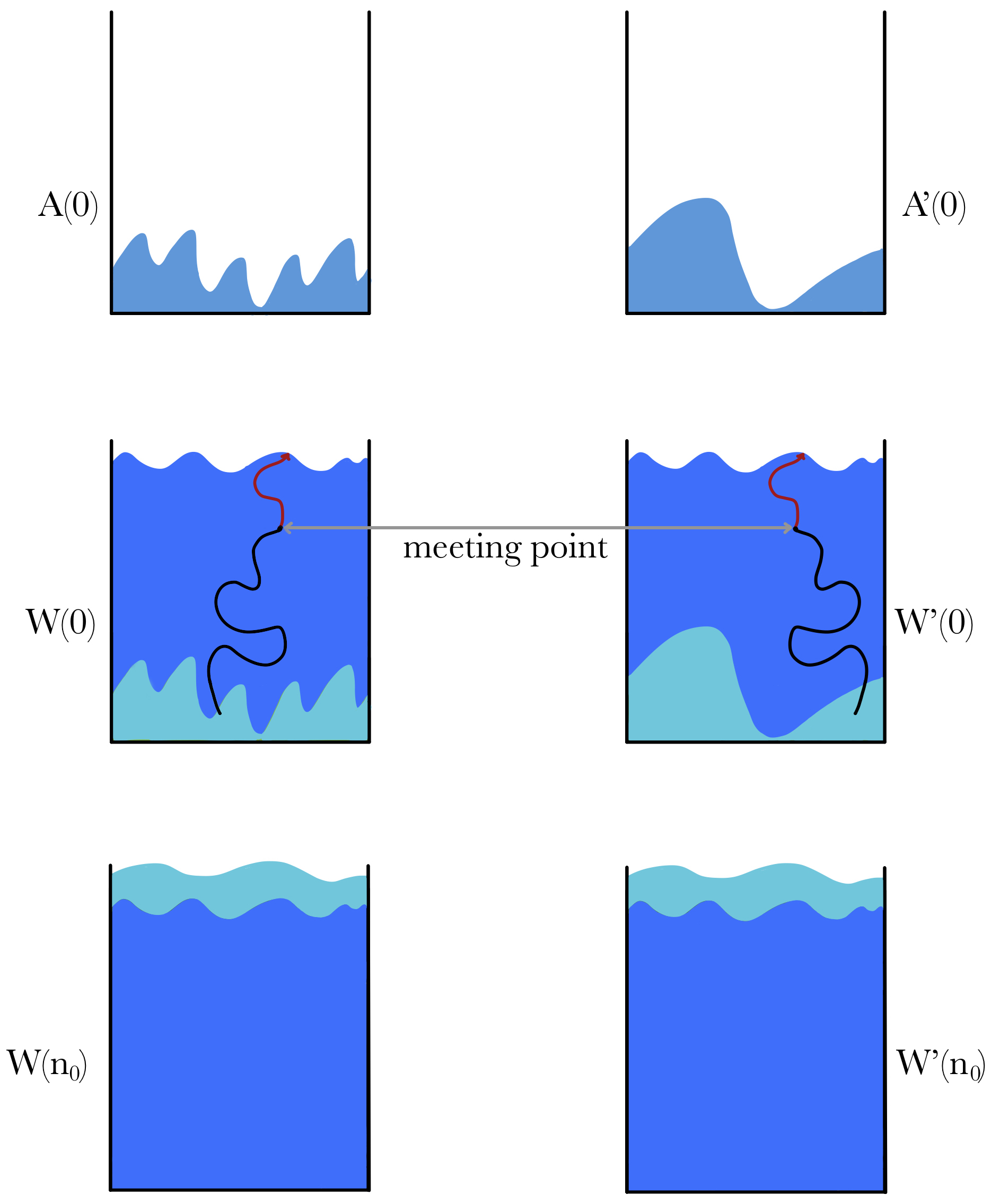}
    \caption{The ``water level coupling'' in the proof of Theorem \ref{th:main2}. Top row: initial clusters. Middle row: the water cluster $W(0)$, with frozen particles in light blue. Frozen particles are released in pairs, and their trajectories (in black) are coupled so that they meet  with high probability before exiting $W(0)$. After meeting, they follow the same trajectory (in red). As a result, they exit in the same exit location.  Bottom row: identical clusters $W(n_0) = W'(n_0)$  resulting from the release~of~all~frozen~particles.} \label{fig:sketch_UB}
\end{figure}
We take $W_0$ to be the initial configuration of two auxiliary processes $(W(t))_{t\leq n_0}$ and $(W'(t))_{t\leq n_0}$, 
that we think of as water flooding the clusters. Water falling in $A_0$ (resp.\ $A'_0$) freezes, and it is only released at a later time. Frozen water particles are released in pairs, and their trajectories are coupled so to make the particles meet with high probability before exiting the respective clusters. 
Clearly, by taking the initial water cluster $W_0$ large enough we can ensure that all pairs of frozen particles meet with high probability before exiting their respective clusters, in which case we have $W(s) = W'(s)$ for all $s\leq n_0$. 
The theorem then follows by invoking the Abelian property (cf.\ Section \ref{sec:Abelian}) for the equalities in law
	\[ A(t_{\g , m} ) \stackrel{(d)}{=} W(n_0) , \qquad 
	A'(t_{\g , m} ) \stackrel{(d)}{=} W'(n_0) . \]
See Figure~\ref{fig:sketch_UB} for an illustration of this argument, and Section \ref{sec:th:main} for the details.

\subsubsection{Theorem \ref{pr:typical_height}: typical clusters are shallow.}
It suffices to prove the result for $\nu_N = \mu_N$. 
Assume Theorem \ref{pr:height_bound}, according to which an IDLA process starting from flat has logarithmic fluctuations for polynomially many steps, with high probability. It suffices to show that such process reaches stationarity in polynomial time to conclude.  This would follow\linebreak from Theorem \ref{th:main2}, if we knew that stationary clusters have at most polynomial height in~$N$, with high probability. 
To see this, we first observe that stationary clusters are dense, since an IDLA process spends only a small amount of time at low density configurations. Here the density of a cluster $A \in \Omega$ is measured via its \emph{excess height} $\cE (A)=h(A) - |A|/N$, that is the difference between the actual height and ideal height of $A$. We show that if the excess height is too large, then it has a negative drift under IDLA dynamics. The advantage of measuring the clusters' density via their excess height lies on the fact that a bound on the latter easily translates on a bound on the number of empty sites below the top level. Once we have such bound, we can try to \emph{fill the holes} below the top level by releasing enough (but at most polynomially many) additional particles. This will leave us with clusters of at most polynomial height. 
Take any deterministic time of the form $t=N^k$ for large enough $k$. Then the cluster $A(t)$ is both stationary and it has at most polynomial height with high probability, which implies that typical clusters have at most polynomial height with high probability, as claimed. 
\begin{figure}[ht!]
  \centering
    \includegraphics[width=0.87\textwidth]{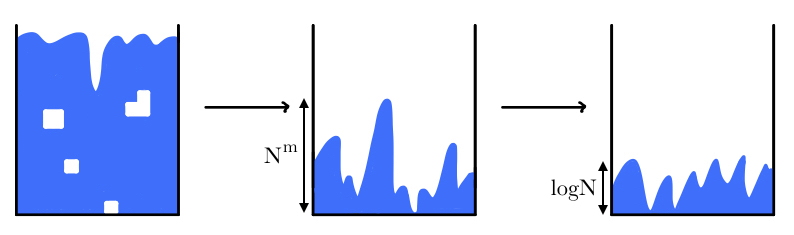}
    \caption{Sketch of the proof of Theorem \ref{pr:typical_height}. Starting from a high density configuration (left), the associated Shifted IDLA process reaches a polynomially high profile (centre), and then a logarithmically high one (right). Both transitions take at most polynomially many steps.} \label{fig:sketch_PR}
\end{figure}

%
%
%
 
\subsection{Organization of the paper}
We start by recalling the Abelian property of Internal DLA in Section~\ref{sec:Abelian}. We then collect some useful preliminary results in Section \ref{sec:preliminaries}. In Section \ref{sec:th:main} we prove Theorem \ref{th:main2}, concerning deterministic initial profiles. 
In Section \ref{sec:height} we bound the fluctuations of IDLA clusters with polynomially many particles (cf.\ Theorem \ref{pr:height_bound}), which we then use in Section~\ref{sec:shallow} to prove Theorem \ref{pr:typical_height}. The upper bound (cf.\ Theorem \ref{th:ZNmain_teo}) is a simple corollary of Theorems \ref{th:main2} and~\ref{pr:typical_height}, as we briefly explain in Section~\ref{sec:proof_upper_bound}. The corresponding lower bound (cf.\ Theorem \ref{th:ZNlower_bound}) is proved in Section \ref{sec:lower_bound}. 
We conclude the paper with a short review of the logarithmic fluctuations result by Jerison, Levine and Sheffield (cf.\ Theorem \ref{th:JLS}), that we include in Appendix \ref{app:survey}.

\subsection*{Acknowledgement}
We are very grateful to Tom Holding, James Norris and Yuval Peres for many valuable comments and suggestions, which substantially improved our results. We are also indebted to the referee for a very careful reading of the paper. 
V.S.\ would like to thank Cornell University, where this work was initiated, for the kind hospitality.

\section{The Abelian property} \label{sec:Abelian}
The Abelian property of Internal DLA was first observed by Diaconis and Fulton \cite{diaconis1991growth}. More recently, it has been used to generate exact samples of Internal DLA clusters in less time than it takes to run the constitutent random walks \cite{friedrich2013fast}. For the related model of \emph{activated random walkers} on $\bbZ$, the Abelian property was used to prove existence of a phase transition \cite{rolla2012absorbing}.

To state the version of the Abelian property that will be used in our arguments, let us start by defining the Diaconis-Fulton \emph{smash sum} in our setting.  Given a set $A \in \Omega$ and a vertex $z \in \bbZ_N \times \bbZ$, define the set $A \oplus \{ z \}$ as follows: 
\begin{itemize}
\item[(i)] if $z \notin A$, then $A \oplus \{ z \} := A \cup \{ z \} $, 
\item[(ii)] if $z \in A$, then $A \oplus \{ z \}$ is the random set obtained by adding to $A$ the endpoint of a simple random walk started at $z$ and stopped upon exiting $A$.  
\end{itemize}
\begin{Remark}
Note that if $A(0) = R_0$ and $\{ z_1 , z_2 , \ldots , z_t \}$ are $t$ independent, uniformly distributed vertices at level zero, then we can build an IDLA cluster $A(t)$ by setting 
	\[ A(t) = ( (A(0) \oplus \{ z_1 \} ) \oplus \{ z_2 \} )\oplus \cdots \oplus \{ z_t \}  . \]
\end{Remark}

The Abelian property, stated below, gives  some freedom on how to build IDLA clusters without changing their laws. 

\begin{Proposition}[Abelian property, \cite{diaconis1991growth}] \label{pr:Abelian}
Given any finite set $\{ z_1 , z_2 , \ldots , z_t \} $ of vertices of $\bbZ_N \times \bbZ$, and a set $A \in \Omega$, the law of 
	\[ ( (A \oplus \{ z_1 \} ) \oplus \{ z_2 \} )\oplus \cdots \oplus \{ z_t \} \]
does not depend on the order of the $z_i$'s. More precisely, if $\sigma: \{ 1 , \ldots , t\} \to \{ 1 , \ldots , t\}$ is an arbitrary permutation of $\{ 1 , \ldots , t\}$, we have the equality in distribution 
	\[  ( (A \oplus \{ z_1 \} ) \oplus \{ z_2 \} )\oplus \cdots \oplus \{ z_t \} 
	\stackrel{(d)}{=} 
	( (A \oplus \{ z_{\s (1)} \} ) \oplus \{ z_{\s (2)} \} )\oplus \cdots \oplus \{ z_{\s (t)} \} .\]
\end{Proposition}

In light of the above result, given $A \in \Omega$ and  a finite subset $B$ of $\bbZ_N \times \bbZ$, we write $A \oplus B$ to denote $((A \oplus \{ z_1 \} ) \oplus \{ z_2 \} ) \oplus \cdots \oplus \{ z_n\}$, where $z_1 , \ldots z_n$ is an arbitrary enumeration of the elements of $B$. 

\begin{Remark}
We point out that Proposition \ref{pr:Abelian} is stated and proved in \cite{diaconis1991growth} for finite sets, while in our setting the set $A$ is infinite. Nevertheless, we can easily reduce to working with finite sets by changing the jump rates of our random walks at level zero to mimic the hitting distribution of a simple random walk after an excursion below level zero. This has the effect of contracting all excursions below level zero to a single step, and it clearly does not change the law of the model. 
\end{Remark}

\section{Preliminaries} \label{sec:preliminaries}

\subsection{A mixing bound}
For $x\in \bbZ_N$, let $P_x^t$ denote the law of a lazy\footnote{A lazy walk stays in place with probability $1/2$, and otherwise takes a simple random walk step.} simple random walk on~$\bbZ_N$ starting from $x$, and denote its stationary measure, uniform on $\bbZ_N$, by $\pi_N$. For $\e >0$ define
	\begin{equation} \label{t_mix}
	 \tau_N (\e ) := \inf \Big\{ t\geq 0 : \max_{x\in \bbZ_N} \| P^t_x - \pi_N \|_{TV} \leq \e  \Big\} .  
	 \end{equation}
The Total Variation (TV) mixing time of this walk is defined to be $\tau_N(1/4)$, which we simply denote by $\tau_N$. It is well known that 
	\begin{equation} \label{a_cy}
	 \t_N \leq N^2 , 
	 \end{equation}
and moreover 
	\begin{equation} \label{t_epsilon}
	 \tau_N ( N^{-\g}) \leq  \big\lceil \log_2 N^\g \big\rceil \tau_N
	\leq 3 \g N^2 \log N   
	\end{equation}
for any $\g >0$ (see e.g.\ \cite{wilmer2009markov}). 
Let $\o = (x(t) , y(t))_{t\geq 0}$  be a simple random walk on $\bbZ_N \times \bbZ$  and define 
	\[ \tau^y_n := \inf \{ t\geq 0 : y(t) = n \} 
	 \]
to be the first time it reaches level $n$.  
The next lemma tells us that by the time the walker has travelled for about $N\log N$ levels in the vertical coordinate, it has mixed well in the horizontal one. 
\begin{Lemma} \label{le:mix}
For all $\g >0$ and all $n \geq 10 \g  N \log N$, it holds 
	\[ \max_{x\in \bbZ_N} \bbP_{(x,0)} (\tau_n^y < 3\g N^2 \log N) \leq N^{-\g }  \]
for $N$ large enough. 
\end{Lemma}
\begin{proof}
Let $(\hat{x} (t))_{t\geq 0} $ be the jump process associated to the motion of the horizontal coordinate, obtained by only looking at $(x(t))_{t\geq 0}$ when the random walk makes a horizontal step. Similarly, let $(\hat{y} (t))_{t\geq 0} $ be the jump process associated to the motion of the vertical coordinate. 
Let
	\[ \hat{\t}^y_n := \inf\{ t\geq 0 : \hat{y} (t) =n \} \]
denote the number of vertical steps made by the walk to reach level $n$. 
For $i\geq 0$, let $G_i -1$ denote the number of moves in the $x$ coordinates between the $i^{th}$ and the $(i+1)^{th}$ move in the $y$ coordinate. Then $(G_i)_{i\geq 0}$ is a collection of i.i.d.\ Geometric 
random variables of mean $2$, and we have 
	\[ \tau^y_n  = \sum_{i=1}^{\hat{\tau}^y_n -1} G_i . \]
It follows that
	\begin{equation} \label{eq:tt*}
	 \begin{split} 
	\max_{x\in \bbZ_N }   \bbP_{(x,0)}  (  \tau^y_n & < N^2 \log N^{3\g}  ) 
	 = \bbP_{0} \bigg( \sum_{i=1}^{\hat{\tau}^y_n -1} G_i   < N^2 \log N^{3\g} \bigg) \\
	& \leq \bbP_{0} \bigg(  \sum_{i=1}^{8 \g N^2 \log N } G_i < 3 \g N^2 \log N  \bigg) + \bbP_0 \big( \hat{\t}^y_n \leq 9 \g N^2 \log N  \big) . 
	\end{split} 
	\end{equation}
We estimate the two terms separately: the first one is controlled by large deviations estimates, while the second one by using the explicit expression for the moment generating function of $\hat{\t}^y_n $. More precisely, we have 
	\[ \begin{split} 
	\bbP_{0} \bigg(  \sum_{i=1}^{8 \g N^2\log N } G_i < 3\g N^2 \log N \bigg) & = 
	\bbP_{0} \bigg( 2 ^{ - \sum_{i=1}^{8 \g N^2 \log N} G_i } >  2^{- 3 \g N^2 \log N }  \bigg)
	\\ &
	\leq  \Big[ 2 \bbE ( 2^{- G_1} )^{\frac{8}{3}} \Big]^{3\g N^2 \log N} 
	\leq  \Big( \frac{2}{9} \Big)^{3 \g N^2 \log N } 
	\leq N^{-3\g} ,
	\end{split}
	\]
where we have used that $\bbE (2^{-G_1}) = 1/3$ for the second inequality.
For the second term, it is simple to check that, for $z \in (0,1)$,
	\[ \bbE_0 \big( z^{\hat{\t}^y_1 } \big) = \frac{1-\sqrt{1-z^2}}{z} , 
	\qquad \bbE_0 \big( z^{\hat{\t}^y_n } \big) 
	= \bbE_0\big( z^{\hat{\t}^y_1 } \big)^n. \]
This gives, for any $z \in (0,1)$,
	\[ \begin{split} 
	\bbP_0 \big( \hat{\t}^y_n < 9 \g N^2 \log N   \big) 
	& \leq \bbE_0 \big( z^{\hat{\t}^y_n } \big)  z^{-9 \g N^2 \log N  } \\
	& = \bigg( \frac{1-\sqrt{1-z^2}}{z} \bigg)^n \cdot \bigg( \frac{1}{z} \bigg)^{9 \g N^2 \log N  } . 
	\end{split} \]
Choose $z\in (0,1)$ of the form $z^2 = 1-1/\a^2$ for some $\a \gg 1$ as $N\to\infty$, to have that 
	\[ \begin{split} 
	\bigg( \frac{1-\sqrt{1-z^2}}{z} \bigg)^n \cdot \bigg( \frac{1}{z} \bigg)^{9 \g N^2 \log N  } 
	& \leq \Big( 1-\frac{1}{\a} \Big)^{n/2} \, \Big( 1-\frac{1}{\a^2} \Big)^{-\frac{9}{2} \g N^2\log N } \\ & 
	\leq \exp \Big( -\frac{n}{2\a} + \frac{5\g }{\a^2} N^2 \log N \Big) .
	\end{split} \]
To have the far r.h.s.\ smaller than, say,  $N^{-5\g/4 }$ it suffices to take 
	\[ n \geq \frac{10 \g }{\a} N^2\log N +  \frac{5}{2}\g \a \log N .\]
Optimizing over $\a$ suggests to take $\a = 2N$, to get $n \geq 10 \g N \log N$. 
\end{proof}

\subsection{The role of starting locations}
The following result tells us that if the walkers have time to mix in the horizontal coordinate before exiting the cluster, the resulting IDLA configuration does not depend too much on their initial positions. 
\begin{Proposition} \label{pr:coupling}
Fix any $T>0 $  such that $T\leq N^m$ for some finite $m \in \bbN $. Let $\{ (x_i , y_i ) \}_{1\leq i \leq T }$ and $\{ (x'_i , y'_i ) \}_{1\leq i \leq T }$ denote two fixed collections of vertices of $\bbZ_N \times \bbZ$ such that $y_i , y_i' \leq 0$ for all $i\leq T$. Let, moreover, $(A(t))_{t\leq T}$ and $(A'(t))_{t\leq T}$ be two IDLA processes with starting configurations $A(0) = A'(0)$, and such that the $i^{th}$ walkers start from $(x_i , y_i)$ and $(x_i' , y_i')$ respectively. Then there exists a coupling of the two processes such that the following holds. For any $\g >0$ there exists a finite constant $C_{\g , m }$, depending only on $\g$ and $m$, such that if 
	\begin{equation}\label{A0}
	 A(0) = A'(0) \supseteq R_{ C_{\g , m } N\log N }
	 \end{equation}
then 
	\begin{equation} \label{AT}
	 \bbP ( A(t) = A'(t) \mbox{ for all }t\leq T ) \geq 1 - N^{-\g }. 
	 \end{equation}
\end{Proposition}
\begin{Remark}
In particular, this tells us that, as long as the IDLA cluster is filled up to level $C_{\g , m } N\log N$, releasing the next $T$ walkers from fixed initial locations below level~$0$ or from uniform locations at level $0$  results in the same final cluster with high probability. 
\end{Remark}
\begin{proof} 
This is an easy consequence of Lemma \ref{le:mix}. 
For $i\leq T$, let $\o_i = ( (x_i(k) , y_i(k) ) )_{k\geq 0}$ and $\o'_i = ( ( x '_i(k) , y'_i(k) ) )_{k\geq 0}$ denote the simple random walk trajectories of the $i^{th} $ walkers starting from $(x_i , y_i )$ and $(x_i' , y_i')$ respectively. 
These are coupled as follows. If, say, $y_i < y_i ' $ then $\o_i '$ stays in place until $\o_i$ reaches level $y_i '$ (and vice versa if $y_i > y_i '$). We can therefore assume that $y_i = y_i '$ without loss of generality. If, moreover, $N$ is even and $|x_i - x_i '|$ is odd, then the first time that $\o_i$ moves in the horizontal coordinate we keep $\o_i '$ in place. Since the probability of $\o_i$ reaching level $N$ before making a horizontal step is $o(2^{-N})$ for large $N$, we can assume that $|x_i - x_i '|$ is even without loss of generality. 

The walks move as follows.
If $\o_i$ moves in the $y$ coordinate, then so does $\o_i '$, and the two walks take the same step. Thus $y_i(0) = y_i '(0) $ implies that $y_i(k) = y_i '(k)$ for all $k\geq 0$. 
If, on the other hand, $\o_i$ moves in the $x$ coordinate, then so does $\o_i '$, and the two walks move according to the reflection coupling on the $N$-cycle (cf.\ Section~\ref{sec:sketch}). 
 Finally, the walkers $\o_i$ and $\o '_i$ stick together upon meeting, that is if $\o_i(k) = \o_i '(k)$ for some $k\geq 0$ then $\o_i(j) = \o_i '(j)$ for all $j\geq k$. Note that if $A(i-1) = A'(i-1)$ and the $i^{th}$ walkers meet before exiting the identical clusters, then $ A(i) = A'(i)$. 

Let, consistently with the notation introduced in the proof of the previous result, $(\hat{x}_i(k))_{k\geq 0} $ and $(\hat{x}_i '(k))_{k\geq 0} $ be the jump processes associated to $(x_i(k))_{k\geq 0}$ and $(x_i'(k))_{k\geq 0}$. 
Then, if 
	\[ \hat{\t}_x := \inf \{ k\geq 0 : \hat{x}_i(k) = \hat{x}'_i(k) \}  \]
we have, by comparison with a simple random walk,
	\begin{equation} \label{optimal}
	\bbP ( \hat{\t}_x > 3\g N^2 \log N ) \leq N^{-\g} 
	\end{equation} 
for all $\g >0$ and $N$ large enough. 
Moreover, by setting 
	\[ \t_x := \inf \{ k\geq 0 : x_i(k ) = x'_i(k  ) \}  , \]
we see that 
$ x_i(\t_x  ) = x'(\t_x ) = \hat{x} ( \hat{\t}_x ) = \hat{x}'(\hat{\t}_x ) $ 
and 	
	\[ \t_x =  \sum_{k=1}^{\hat{\t}_x -1} G_i , \]
for $(G_i)_{i\geq 1}$ i.i.d.\ Geometric random variables of mean $2$. Let $\g ' = 6(\g +m )$ and $\g '' = \g '/3 $. Recall that $R_n$ denotes the infinite rectangle of height $n$, and assume that $A(0) = A'(0) \supseteq R_n$ for $n =10 \g ' N \log N$. Denote  by $\t_n$ the first time both walkers reach level $n$. 
Then by Lemma~\ref{le:mix} and \eqref{optimal} we have 
	\[ \begin{split} 
	\bbP ( \o_i \mbox{ and } \o '_i & \mbox{ exit } R_n \mbox{ before meeting} ) 
	 = \bbP ( \t_x > \t_n ) 
	= \bbP \Big( \sum_{k=1}^{\hat{\t}_x -1} G_k > \t_n \Big) 
	\\ & 
	\leq \bbP \bigg( \sum_{k=1}^{\hat{\t}_x } G_k > 3\g ' N^2 \log N \bigg) 
	+ \bbP (\t_n < 3 \g ' N^2 \log N ) 
	\\ & \leq \bbP \bigg( \sum_{k=1}^{3 \g '' N^2 \log N} G_k 
	> 3 \g ' N^2 \log N  ) \bigg) + 
	\bbP \big( \hat{\t}_x > 3 \g '' N^2 \log N\big)
	+ N^{-\g ' }
	\\ & \leq \bbP \bigg( \sum_{k=1}^{3 \g '' N^2 \log N  } G_k 
	> 9 \g '' N^2 \log N  \bigg) + 
	 N^{-\g '' } + N^{-\g '}   
	\\ & \leq 
	\bigg[ \frac{\bbE ( e^{\l G_k } ) }{e^{3\l}} \bigg]^{ 3\g '' N^2 \log N } + 2N^{-\g ''}, 
	\end{split}
	\]
where the third inequality follows from \eqref{optimal}.
Finally, taking $\l = \log (3/2) >0$ makes the term in the square brackets equal to $8/9$, from which we conclude that 
	\[ 
	\bbP ( \o_i \mbox{ and } \o '_i  \mbox{ exit } R_n \mbox{ before meeting} ) 
	 \leq N^{-3 \g '' N^2 \log \frac{9}{8}  }  + 2 N^{-\g ''} 
	  \leq N^{-(\g +m)} \]
for $N$ large enough.  
In all, we have found that 
	\[ \begin{split} 
	\bbP ( \exists t \leq T \mbox{ such that } A(t) \neq A'(t) ) 
	& \leq \bbP ( \exists i \leq T \mbox{ such that } 
	\o_i \mbox{ and } \o_i ' \mbox{ exit } R_n \mbox{ before meeting} ) 
	\\ & 
	\leq T N^{-(\g + m ) } \leq N^{-\g } . 
	\end{split}\]
This shows that we can take $C_{\g , m} = 60( \g + m )$ in \eqref{A0} to have \eqref{AT}, thus concluding the proof. 
\end{proof}


\section{The water level coupling} \label{sec:th:main}
In this section we prove Theorem \ref{th:main2}. 
\begin{proof}[Proof of Theorem \ref{th:main2}]
Let $A_0 , \, A'_0 $ be any two clusters in $\Omega$ with $|A_0| = |A'_0| = n_0$ and such that 
	\[ h_0 = \max\{ h(A_0) , h(A'_0)\} \leq N^m . \]
Let $C_{\g +1 , m+1}$ and $b_{\g+1, m+2}$ be defined as in Proposition \ref{pr:coupling} and Theorem~\ref{pr:height_bound} respectively, and define 
	\[ d'_{\g , m} := 2 \max \{ C_{\g +1 , m+1} ,  \, b_{\g+1, m+2} \} . \] 
We build an auxiliary water cluster $W_0$ by adding $t_{\g , m } = h_0 N + d'_{\g , m } N^2 \log N $ particles to the flat configuration $R_0$ according to IDLA rules. Then, since $t_{\g , m} \leq N^{m+2}$, by Theorem~\ref{pr:height_bound} we have 
	\[ 	 \bbP \Big(  W_0 \supseteq R_{\frac{t_{\g , m}}{N} - b_{\g +1 ,m+2}  \log N }
	 \Big) \geq 1-N^{-(\g +1)} \]
for $N$ large enough. In particular, since $d'_{\g , m} \geq C_{\g +1 , m+1} + b_{\g +1 ,m+2}$, this gives 
	\[ 	\bbP \Big(  W_0 \supseteq R_{h_0 + C_{\g +1 ,m+1}  N\log N }
	 \Big) \geq 1-N^{-(\g +1)} , \]
i.e.\ the water cluster is completely filled up to height $h_0 + C_{\g +1 ,m+1}  N\log N$ with high probability. 
Write 
	\[ A_0 = \{ z_1 , z_2 , \ldots , z_{n_0} \} , \qquad 
	A'_0 = \{ z'_1 , z'_2 , \ldots , z'_{n_0} \}  \]
for arbitrary enumerations of the sites in $A_0$, $A'_0$ above level $0$. We 
define two auxiliary processes  $(W(t))_{t\leq n_0} $, $(W'(t))_{t\leq n_0}$ by setting $W(0) = W'(0) = W_0$ and inductively defining for $t\leq n_0$
	\[ W(t) = W(t-1) \cup \{ Z_t \} , \qquad W'(t) = W'(t-1) \cup \{ Z'_t \}, \]
where $Z_t$, $Z'_t$ denote the exit locations from $W(t-1)$, $W'(t-1)$ of simple random walks on $\bbZ_N \times \bbZ$ starting from $z_t$, $z'_t$. These walks are coupled as follows. If $z_t$ is at a lower level than~$z'_t$, then the walk starting from $z_t$ moves freely (independently of everything else) until it reaches the level of $z'_t$, while the other walk stays in place. Once at the same level, the walks move together in the vertical coordinate, whereas the horizontal coordinates evolve according to the \emph{reflection coupling}: if one steps to the left, the other one steps to the right, and vice versa\footnote{Here we again use the first horizontal step of the walks to adjust the parity of the difference of the horizontal coordinates, if needed, as explained in the proof of Proposition \ref{pr:coupling}.}. 
Once the walks meet, they move together in both coordinates.
 Since $n_0 \leq h_0 N \leq N^{m+1}$, and all the walks start at distance at least $C_{\g +1 , m+1}N \log N$ from the boundary of the cluster,  Proposition \ref{pr:coupling} gives\footnote{Although the particles' starting positions were taken to be below level $0$ in Proposition \ref{pr:coupling} for notational convenience, the result applies in this setting by invariance under vertical shifts.}
	\[ \begin{split} 
	\bbP \big( W(n_0) \neq W'(n_0) \big) 
	\leq & \, \bbP \Big( W(n_0) \neq W'(n_0) \Big| W_0 \supseteq  R_{\frac{t_{\g , m}}{N} + C_{\g +1 ,m+1}  \log N } \Big) + \\ & + \bbP \Big( W_0 \nsupseteq  R_{\frac{t_{\g , m}}{N} + C_{\g +1 ,m+1}  \log N } \Big) 
	\leq 2N^{-(\g + 1)} \leq N^{-\g} 
	\end{split} \]
for $N$ large enough. The result then follows by observing that if $(A(t))_{t\geq 0}$ and $(A'(t))_{t\geq 0}$ denote two IDLA processes starting from $A_0$ and $A'_0$ respectively, then 
	\[ A(t_{\g , m} ) \stackrel{(d)}{=} W(n_0) , \qquad 
	A'(t_{\g , m} ) \stackrel{(d)}{=} W'(n_0)  \]
by the Abelian property. Indeed, if we denote by $w_1 , w_2 , \ldots , w_{t_{\g , m}}$ the starting locations of the $t_{\g , m}$ walkers used to grow $W_0$, then, with the notation introduced in Section \ref{sec:Abelian}, we have 
	\[ W(n_0) = ((( R_0 \oplus \{ w_1 \} ) \oplus \{ w_2 \} )\oplus \cdots \oplus 
	\{ w_{t_{\g , m}} \} ) \oplus A_0 , \]
while 
	\[ A(t_{\g , m} ) = (( R_0 \oplus A_0 ) \oplus \{ w_1 \} ) \oplus \{ w_2 \} ) \oplus \cdots \oplus \{ w_{t_{\g , m}} \} .  \]
The claimed equality in law is then given by Proposition \ref{pr:Abelian}. 
\end{proof}
\begin{Remark}
Note that to prove Theorem \ref{th:main2} we have constructed a coupling of the final clusters $A(t_{\g ,m})$, $A'(t_{\g ,m})$, not of the whole processes $(A(t))_{t\geq 0}$, $(A'(t))_{t\geq 0}$. 
\end{Remark}

\section{Logarithmic fluctuations for large clusters} \label{sec:height} 
In this section we bound the fluctuations of an IDLA cluster with polynomially many particles, thus proving Theorem \ref{pr:height_bound}. 
To start with, we claim that the following~holds. 
\begin{Theorem}\label{th:JLS}
Let $(A(t))_{t\geq 0}$ denote an IDLA process on $\bbZ_N \times \bbZ$ starting from the flat configuration $A(0) = R_0$. Fix any $T\leq (N \log N)^2$. 
Then for any $\g >0$ there exists a finite constant $a_\g$, depending only on $\g$, such that 
	\[ \bbP \Big( R_{\frac{T}{N} -a_\g \log N}  \subseteq A(T) \subseteq R_{\frac{T}{N} + a_\g \log N } 
	\Big) \geq 1 - N^{-\g}  \]
for $N$ large enough. 
\end{Theorem}
The above theorem is stated for $T=N^2$ in \cite{jerison2014internal2}, where the authors use it to show convergence of space-time averages of IDLA fluctuations to the Gaussian free field. It can be proved by the same arguments used in \cite{jerison2012logarithmic} for planar IDLA, which are in fact rather simplified by the structure of the cylinder graph. 
Since this proof does not appear anywhere, and since we need to extend it to larger values of$~T$, we give it in Appendix~\ref{app:survey}. 

Assuming Theorem \ref{th:JLS}, we can proceed with the proof of Theorem \ref{pr:height_bound}.



\begin{proof}[Proof of Theorem \ref{pr:height_bound}]
It suffices to show that \eqref{eq:height_bound} holds for any fixed $T \leq N^m$. The result will then follow by replacing $\g$ with $\g+m$ and using the union bound over all $T\leq N^m$.

If $T = \cO (N^2 \log N )$ then we are done by Theorem \ref{th:JLS}, so assume $T\gg N^2 \log N$. We are going to iteratively reduce the size of the cluster, until it becomes $\cO (N^2 \log N)$. To this end, 
let $a_{\g+2+m}$ and $C_{\g+2 , m}$ denote the constants in Theorem \ref{th:JLS} and Proposition \ref{pr:coupling} respectively, and set $a:= a_{\g+2+m}$, $c:= C_{\g+2 , m}$ for brevity. At cost of increasing these constants, we can assume that $2c \geq 1$ and both $a \log N$ and $c  \log N$ are integers. Define 
	\[ n:= a N \log N + 2c N^2 \log N . \]
We build a large cluster, with the same law of $A(T)$, by using the water processes mentioned in the introduction. To this end, let $W_1$ denote the cluster obtained by adding $n$ particles to $A(0)=R_0$ according to IDLA rules. Then by Theorem \ref{th:JLS} 
	\begin{equation}\label{eq:block1}
	 \bbP \big( R_{2cN\log N} \subseteq W_1 \subseteq R_{2a \log N + 2cN\log N} \big)
	  \geq 1-N^{-(\g+2+m)} 
	 \end{equation}
for $N$ large enough. Let  
	\[ W^f_1 := W_1 \cap R_{2cN\log N} \]
denote the region which is filled with high probability after $n$ releases.  Write further  
	\[ F_1 := W_1 \setminus W_1^f \]
for the fluctuation region. Water particles in the fluctuation region are declared frozen, and they will be released at a later time.
On top of the water-filled region $W_1^f$ we are going to build a second cluster with again $n$ particles, so to fill a rectangle of height $4cN\log N$ with high probability.  Let $W_2$ denote such a cluster,  obtained by adding $n$ particles to $W_1^f$ according to IDLA rules (thus, new water particles can, and will, settle inside $F_1$). Write $W_2^f$ and $F_2$ for the filled region and fluctuation region of such cluster. Then, as in~\eqref{eq:block1}, we have 
	\[ \bbP \Big( R_{4cN\log N} \subseteq 
	W_1^f \cup W_2 \subseteq R_{2a N \log N + 4cN\log N} \Big) \geq 1-2N^{-(\g +2+m)} . \]
We again declare particles in $F_2$ frozen, and treat their locations as empty for subsequent walkers. This procedure is iterated for $k= \lfloor T/n \rfloor -1$ rounds.

Let $\Omega_1$ denote the event that the fluctuations bound \eqref{eq:block1} holds for all $k $ rounds, that is 
	\[ \Omega_1 := \big\{ W_1^f \cup W_2^f \cup \ldots \cup W_k^f  = R_{2kcN\log N} \} 
	\cap \bigcap_{l=1}^k \big\{ F_l \subseteq 
	R_{2a\log N + 2lcN\log N} \setminus  R_{2lcN\log N} \big\} . \]
Then by \eqref{eq:block1}
	\[ \bbP ( \Omega_1 ) \geq 1-kN^{-(\g +2+ m)} \geq 1-N^{-(\g +2 )} .\]
\begin{figure}[h!]
  \centering
    \includegraphics[width=.9\textwidth]{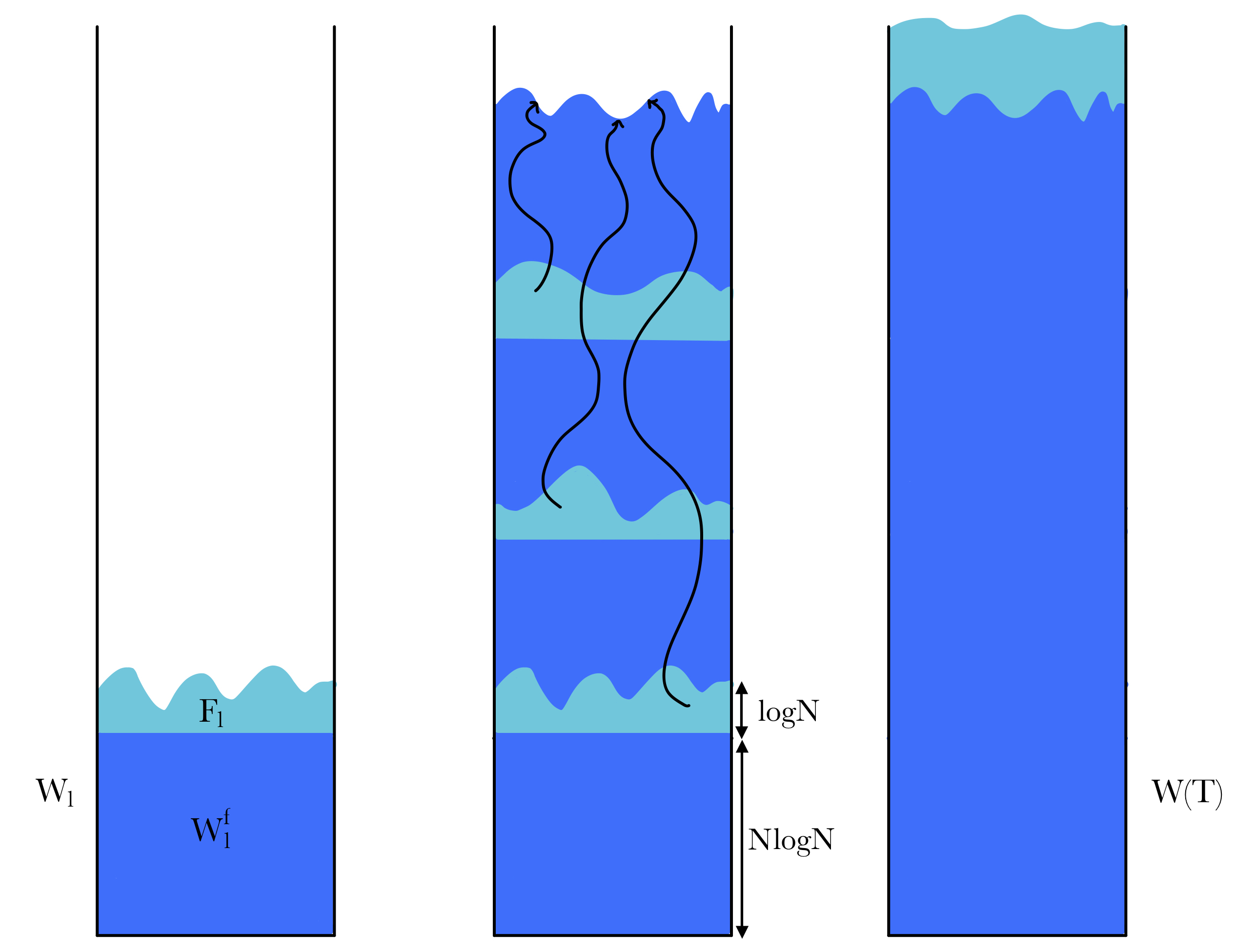}
    \caption{Sketch of the proof of Theorem \ref{pr:height_bound}.}
\end{figure}
Let us restrict to this good event. It remains to release the $akN\log N$ frozen particles from their locations, plus $T- k n$ new ones uniformly from level~$0$. Denote by $W(T)$ the cluster obtained after all the released particles have settled, so that $|W(T)|=T$. By the Abelian property, $W(T)$ has the same distribution as $A(T)$. 
Write $T' = T-kn + akN\log N $ for brevity, 
and let $W'(T')$ denote the cluster obtained by adding $T'$ particles to the filled rectangle $R_{2kcN\log N}$ according to IDLA rules (more precisely, we start $T'$ random walks  uniformly from level zero, independently of everything else, and add their exit locations to the cluster). We argue that we can couple $W(T)$ and $W'(T')$ so that they coincide with high probability. To see this, we proceed as follows. First release $T-kn$ new particles uniformly from level $0$, and note that, since $T-kn \geq n$, on the event $\Omega_1$ these will fill a rectangle of height $2cN\log N$ with probability at least $1-N^{-(\g +2+m)}$. If this happens, then all the frozen particles are at distance at least $cN\log N$ from the boundary of the cluster. Thus by Proposition \ref{pr:coupling} we can couple $W(T)$ and $W'(T')$ so that 
	\[  \bbP \Big( W(T) = W'(T') \Big) 
	\geq 1- \bbP (\Omega_1) - N^{-(\g +2+m)} - N^{-(\g +2) } \geq 1 - N^{-(\g +1 )} . \]
In all, we have reduced the problem of bounding the fluctuations of $W(T)\stackrel{(d)}{=}A(T)$ to the same one for the smaller cluster 
	\[ A(T') \stackrel{(d)}{=} W'(T') \setminus R_{2kcN\log N} , \]
at the price of a small probability of failure. Since $T' \leq 2 \max\{ 2n , aT/N \}$, this either makes the number of particles $\cO ((N \log N )^2)$, in which case we stop, or it decreases it by a multiplicative factor $2a/N$. Thus after at most $m$ iterations of the above procedure we are back to clusters with $\cO ((N \log N)^2) $ particles, which we know to have logarithmic fluctuations. This shows that 
	\[ 	 \bbP \Big(  R_{\frac{T}{N} - a \log N }
	\subseteq  A(T) \subseteq
	R_{\frac{T}{N} + a  \log N  } \Big) \geq 1-m N^{-(\g +1 ) }
	\geq 1-N^{-\g} , \]
so \eqref{eq:height_bound} holds with $b_{\g , m} = a$, as wanted.

\end{proof}

\section{Typical profiles are shallow} \label{sec:shallow}
In this section we show that typical IDLA profiles have at most logarithmic height, thus proving Theorem \ref{pr:typical_height}.  This is achieved by combining Theorem~\ref{th:main2} with a control on the density of stationary clusters. 

\subsection{Decay of the excess height} \label{sec:decayE}
We distinguish between high and low density clusters by looking at their excess height, that we now define. 
Recall that $\Omega$ denotes the set of clusters completely filled up to level $0$, so that  $h(A) \geq 0$ for all $A\in \Omega$, while $|A|$ denotes the total number of sites in $A$ strictly above level $0$.  
\begin{Definition}[Excess height]
For $A\in\Omega$, the \emph{excess height} of $A$, denoted by $\cE (A)$,  is defined as the difference between the height of $A$ and the minimum possible height, i.e.,\ 
	\[ \cE(A) := h(A) - \frac{|A|}{N}  . \] 
\end{Definition}
Note that $\cE(A) \geq 0 $. We say that a cluster $A$ has high density if $ \cE (A) \leq \cE^*$, where $\cE^* $ is a constant to be chosen later depending only on $N$. If instead $\cE (A) > \cE^*$, then $A$ is said to have low density. 

\begin{figure}[h!]
  \centering
    \includegraphics[width=0.45\textwidth]{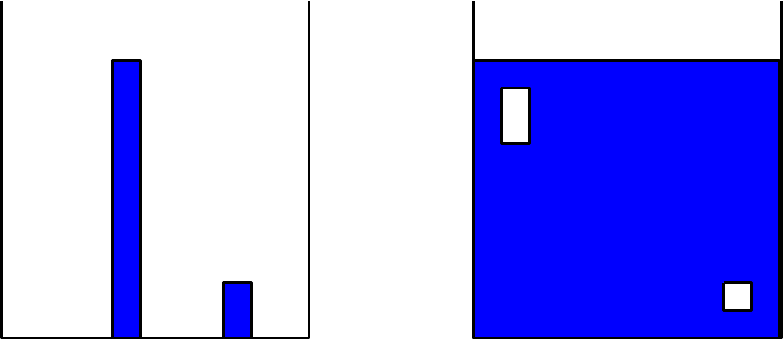}
    \caption{Clusters with high excess height (left) and low excess height (right).} 
\end{figure}

To start with, we prove that, for a suitable choice of $\cE^*$, the excess height drops below~$\cE^*$ quickly under the IDLA dynamics.

\begin{Lemma}\label{le:excess_height}
Let $(A(t))_{t\geq 0}$ denote an IDLA process on $\bbZ_N\times \bbZ$ with $A(0) \supseteq R_0$. For $t\geq 0$ let $\cE(t) := \cE(A(t))$ denote the excess height of $A(t)$. 
Then for any $\eta \in (0,1)$ there exists a constant $\cE^* = \cE^* (N, \eta )$, depending only on $N$ and $\eta$,  such that if 
	\[ T_{\cE^*}:= \inf \{ t\geq 0 : \cE (t) \leq   \cE^* \} \]
denotes the first time that the excess height drops below $\cE^*$, then 
	\[ \bbP( T_{\cE^*} >t ) \leq  e^{-\frac{\eta^2}{8 N^2} t} \,,\]
for $t> \frac{N}{\eta} (\cE(0) - \cE^*)$ and $N$ large enough.
\end{Lemma}

Lemma \ref{le:excess_height} is an easy consequence of the following result, which tells us that if a cluster has low density then its excess height has a negative drift under the IDLA dynamics.

\begin{Lemma}\label{le:towards_height}
Let $\cF_t := \sigma \{ A(s) : s\leq t \}$, and write $h(t)$ in place of $h(A(t))$ for brevity. Then for all $\eta \in (0, 1 )$ there exists a constant $\cE^* = \cE^* (N ,\eta )$, depending only on $N$ and $\eta$, such that if  $\cE (t) >  \cE^* $ then  it holds 
	\[ \bbE \big( h(t+1) - h(t) | \cF_t \big) < \frac{1-\eta }{N} \, \]
for $N$ large enough.
\end{Lemma}
\begin{proof}
Fix $\eta \in (0,1)$ throughout. 
For $k\leq h(t)$, we say that level $k$ is \emph{bad} if it contains at least one empty site, that is if $A(t)^c \cap \{ y=k\} \neq \emptyset$. 
We claim that if $\cE (t) \geq  \cE^*$ then there are at least $\cE^*$ bad levels between $0$ and the top one $h(t)$. Indeed, since $h(t) \geq \frac{|A(t)|}{N} +\cE^*$ then there are at least $\frac{|A(t)|}{N} +\cE^*$ levels above level $0$ in the cluster, and at most $\big\lfloor \frac{|A(t)|}{N} \big\rfloor$ of them can be completely filled. 

Recall from \eqref{t_mix} the definition of $\t_N(N^{-\g})$, and note that by \eqref{t_epsilon} and  Lemma \ref{le:mix} with $\g =2 $ we have 
	\[  \max_{x\in \bbZ_N} \bbP_{(x,0)} ( \tau_n^y < \tau_N (N^{-2}) ) 
	\leq \max_{x\in \bbZ_N} \bbP_{(x,0)} ( \tau_n^y < 6 N^2 \log N ) \leq N^{-2} \] 
as long as $n\geq n^*:= 20 N \log N$. 
Then, if $\o = (x(k) , y(k) )_{k\geq 0}$ is a simple random walk on $\bbZ_N\times \bbZ$, the above implies that 
	\[ \min_{x,x'\in \bbZ_N}  \bbP_{(x,0)} \big( y (\tau_{n^*}^y ) = (x',n^* ) \big) 
	\geq \frac{1}{N} - \frac{1}{N^2} 
	\geq \frac{1}{2N} \]
for $N$ large enough. 
In words, a simple random walk on $\bbZ_N \times \bbZ$ starting at level zero has probability at least $1 /2N$ to reach  level $n^*$ for the first time at any given vertex, uniformly over the starting location. Now, 
since there are at least $\cE^*$ bad levels, we can find at least $\frac{\cE^*}{n^*}$ bad levels at distance at least $n^*$ from each other. We treat these as traps. More precisely, for a new particle to increase the height of the cluster $A(t)$, the particle must travel through all the bad levels without exiting the cluster, until it reaches the top. Since we are taking the bad levels sufficiently far apart, the particle has time to mix in between, so whenever it reaches a bad level for the first time it has probability at least $\frac{1}{2N}$ to fall outside the cluster. By taking enough bad levels, then, we can make the probability for the particle to survive all of them arbitrarily small. 
\begin{figure}[h!]
  \centering
    \includegraphics[width=0.25\textwidth]{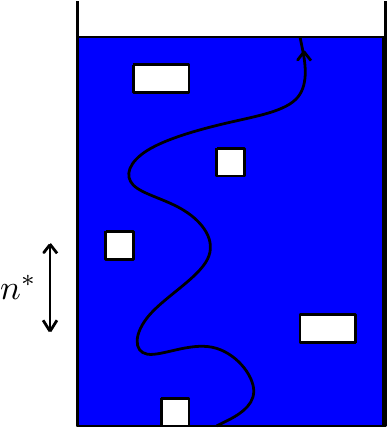}
    \caption{An example of the trajectory of a particle surviving all the bad levels at distance at least $n^*$ one from the other. } 
\end{figure}

Formally, since the walker has probability at least $\frac{1}{2 N}$ to exit the cluster upon reaching a bad level for the first time, we have 
	\begin{equation} \label{chooseC}
	 \bbP (h(t+1) - h(t) =1 | \cF_t ) \leq \Big( 1-\frac{1}{2N}  \Big)^{\cE^*/ n^*}  \,.
	 \end{equation}
To make this smaller than $\frac{1-\eta}{N}$ it suffices to take $\cE^*$ large enough, precisely 
	\begin{equation}\label{eq:E*}
	\cE^* \geq \bigg\lceil 2 n^* N \log \Big( \frac{N}{1-\eta } \Big)  \bigg\rceil , 
	\end{equation}
which concludes the proof.
\end{proof}

Note in particular that 
	\begin{equation} \label{e*}
	 \cE^* > 40 (N \log N)^2 , 
	 \end{equation}
which will be useful later on. We can now prove that low density clusters tend to decrease their excess height under IDLA dynamics. 

\begin{proof}[Proof of Lemma \ref{le:excess_height}]
Let $n_0 = |A(0) |$ denote the number of sites in $A(0)$ of positive height, and assume that $\cE(0) = h(0) - n_0 /N > \cE^*$, with $\cE^*$ as in \eqref{eq:E*}. 
It follows from Lemma~\ref{le:towards_height} that 
	\[ M(t) := h(t) - \frac{n_0 +(1-\eta)t}{N} , \quad t\geq 0\] 
is a supermartingale up to the stopping time $T_{\cE^*}$, with $M(0) = \cE(0)$. 
As a consequence, the stopped process $M^*(t) = M(t\wedge T_{\cE^*})$ is a supermartingale for all $t\geq 0$, with $M^*(t) = M(t)$ for $t\leq T_{\cE^*}$, and 
	\[ |M^*(t+1) - M^*(t) | 
	\leq  |h(t+1) - h(t) | + \frac{1-\eta}{N} \leq 2 \]
for all $t\geq 0$.  
Now, since  $|A(t)|= n_0+t$, for $t< T_{\cE^*}$ we have
	\[ 
	\{ \cE (t) >\cE^* \}  = \Big\{ h(t) - \frac{n_0+t}{N} > \cE^* \Big\} 
	 \subseteq\Big\{ M(t) >\cE^* + \frac{\eta}{N} t \Big\} 
	 = \Big\{  M^*(t) >\cE^* + \frac{\eta}{N} t \Big\} .
	\]
It thus follows from Azuma's inequality that, for $t> \frac{N}{\eta} (\cE(0) - \cE^*)$, 
	\[\begin{split} 
	 \bbP(T_{\cE^*} >t) & \leq \bbP \big( T_{\cE^*} >t , \cE(t) > \cE^* \big) 
	 \leq \bbP \big( T_{\cE^*} >t , M^*(t) >\cE^* + \frac{\eta}{N} t \big) \\ & 
	 \leq \bbP \Big( M^*(t) - M^*(0) > \cE^* + \frac{\eta}{N} t  - M(0) \Big) \\
	& \leq \exp \Big\{ - \frac{ ( \cE^* + \eta t /N -M(0))^2}{8t} \Big\} 
	\\ & \leq \exp \Big\{ - \frac{ ( \cE^*  - \cE(0) )^2+ (\eta t /N)^2 }{8t} \Big\} 
	\\ & \leq \exp \Big\{ - \frac{\eta^2}{8 N^2} t \Big\} \, , 
	\end{split} \]
as claimed. 
\end{proof}

\begin{Remark} \label{remark_pr} 
Lemma \ref{le:excess_height} implies that Shifted IDLA is positive recurrent, thus proving the existence of the stationary distribution $\mu_N$. To see this it suffices to show, for example, that the flat configuration $R_0$ is positive recurrent. 
Let $ (A (t))_{t\geq 0}$ be an IDLA process with $A (0) = R_0$, and define  $T_0$ to be the first time $t$ such that $A (t) = R_k$ for some $k\geq 1$. Here is a wasteful way to show that $\bbE_{R_0} (T_0) <\infty$. Fix any $\eta \in (0,1)$ and let $\cE^*=\cE^*(N, \eta )$ be the integer constant in Lemma \ref{le:excess_height}. Starting from $R_0$, release $\cE^* N$ particles: with probability at least $N^{-\cE^* N}$ the final configuration will be the filled rectangle $R_{\cE^*}$. If not, then the excess height has increased by at most $\cE^* N$. We keep releasing particles until the excess height falls again below $\cE^* $: by Lemma \ref{le:excess_height} this takes a random time with exponential tails, and hence finite expectation. Once the excess height is at most $\cE^*$, there are at most $\cE^* N$ empty sites below the top level in the cluster. We release as many particles as the number of the empty sites below the top level: with probability at least $N^{-\cE^* N}$ the final configuration will be a filled rectangle $R_k$ for some $k\geq 1$. If not, the excess height is at most $\cE^* + \cE^* N$: we again wait for it to fall below $\cE^*$, and iterate. After at most a Geometric$\big(N^{-\cE^* N} \big)$ number of attempts, the final configuration will be a filled rectangle. Each attempt takes $\cE^* N$ releases, plus the time it takes for the excess height to fall below $\cE^* $ starting from $\cE^* +\cE^* N$, which has exponential tails. 
In all, we conclude that the total time it takes to go from $R_0$ back to a flat configuration $R_k$ for some $k\geq 1$ has finite expectation. 
\end{Remark}

\subsection{Typical clusters are dense}
In this section we show that, for $N$ large enough, $\mu_N$ gives high probability to high density clusters. This shows that stationary clusters are dense, and hence that typical clusters are dense, with high probability.

\begin{Proposition}\label{pr:initial_excess}
For $\cE^*=\cE^*(N , \eta  )$ as in 
\eqref{eq:E*} and $N$ large enough, it holds 
	\[ \mu_N \Big( \big\{ A :  \cE (A) > 2 \cE^*  \big\} \Big)  \leq e^{-N/2}  . \]
\end{Proposition}

\begin{proof}
Recall that for $(A(t))_{t\geq 0}$ IDLA process on $\bbZ_N\times \bbZ$, we denote by $(A^*(t))_{t\geq 0}$ the associated shifted process. Let us denote by $\Omega^* := \{ A^* : A\in \Omega \}$ its state space.   

Define the sets 
	\[ \cA = \{ A \in \Omega^* : \cE (A) \leq \cE^* \} , \qquad 
	 \cB = \{ A \in \Omega^* : \cE (A) \leq 2\cE^* \} .\]
We seek to bound $\mu_N(\cB^c)$. By the Ergodic theorem for positive recurrent Markov chains, 
	\begin{equation}\label{ergodic}
	 \mu_N (\cB^c ) = \lim_{t\to\infty } \frac{1}{t} \bbE \bigg( \sum_{i=1}^t \mathds{1} (A^*(i) \in \cB^c ) \bigg) .
	 \end{equation}
Define the following stopping times: 
	\begin{align*} 
	 &\t_{\cA ,\partial \cB }^{1} := \inf \{ t\geq 0 : A^*(t) \in \partial \cB \}  , 
	 \\ & 
	 \t_{\partial \cB , \cA }^{1} := \inf \{ t\geq  \t_{\cA ,\partial \cB }^{1} : A^*(t) \in \cA\}  - 
	\t_{\cA ,\partial \cB }^{1}  ,
	\\  & \t_{\cA ,\partial \cB }^{i} := \inf \{ t\geq \t_{\partial \cB , \cA }^{i-1} : A^*(t) \in \partial \cB \}  
	- \t_{\partial \cB , \cA }^{i-1} ,  \quad i \geq  2, \\  &
	\t_{\partial \cB , \cA }^{i} := \inf \{ t\geq  \t_{\cA ,\partial \cB }^{i} : A^*(t) \in \cA\} 
	- \t_{\cA ,\partial \cB }^{i} , \quad i\geq 2,
	\end{align*}
where $\partial \cB = \{ A \in \Omega^* : \cE(A) \in [2\cE^* , 2\cE^*+1)\}$. 
Note that, since the excess height always changes by either $1-1/N$ or $1/N$,
	\[ \cE \bigg( \sum_{i=1}^k \big(  \t_{\cA ,\partial \cB }^{i}  + \t_{\partial \cB , \cA }^{i} \big) \bigg) \leq \cE^*+1 ,
	\qquad \; \forall \; k\geq 1. \]
We divide the interval $[0,t]$ into excursions from $\cA$ to $\partial \cB$ and then from $\partial \cB$ to $\cA$. Since during excursions from $\cA$ to $\partial \cB$ the process is in $\cB$, only excursions from $\partial \cB$ to $\cA$ contribute to the expectation in \eqref{ergodic}. Let us say that the concatenation of an excursion from $\cA$ to $\partial \cB$ and from $\partial \cB$ to $\cA$ is a \emph{complete excursion}. In the next lemma we bound the number of complete excursions by time $t$. 
\begin{Lemma} \label{le:Kt}
Let 
	\[ K(t) := \sup \Big\{ k\geq 1 : \sum_{i=1}^k \big(  \t_{\cA ,\partial \cB }^{i}  + \t_{\partial \cB , \cA }^{i} \big) \leq t \Big\} \]
denote the number of complete excursions by time $t$. 
Then for $\g = \frac{6}{N e^{N}  } $ and $N$ large enough  it holds 
	\begin{equation} \label{gamma}
	\bbP (K(t) \geq \g t) \leq \exp \Big( - \frac{t}{N e^{N} }\Big) . 
	\end{equation}	 
\end{Lemma}
\begin{proof}
Let 
	\[ \tilde{K}(t) := \sup \Big\{ k\geq 1 : \sum_{i=1}^k   \t_{\cA ,\partial \cB }^{i}  \leq t \Big\} .\]
Then clearly $\tilde{K}(t) \geq K(t)$, and so $\bbP (K(t) \geq \g t ) \leq \bbP (\tilde{K} (t) \geq  \g t)$. Moreover, since the excess height changes by at most~$1$  at each step, at the start of each excursion from $\cA$ to $\partial \cB$ the excess height must lie between $\cE^*-1 $ and $\cE^*$. We know (cf.\ Lemma \ref{le:towards_height}) that when the excess height is greater than $\cE^*$ it has a negative drift, but we do not have any information on its drift when the process is in the set $\cA$. To overcome this problem, we simply ignore the time spent in $\cA$. Indeed, we will see that the time spent in $\cB \setminus \cA$ is large enough to give us what we want.

We seek to stochastically bound $\t^i_{\cA , \partial \cB}$ from below. To this end, define the following auxiliary random walk on $\frac{1}{N} \bbN =\big\{ \frac{n}{N} : n\in \bbN \big\}$, with a reflecting barrier at zero: 
	\begin{equation}\label{rwX}
	 \begin{split}
	& X_0 = 0 , \qquad \bbP \Big( X_{i+1} =1 -\frac{1}{N} \, \Big| X_{i} =0 \Big) =1 , \\
	& X_{i+1} - X_i = \begin{cases}
	1-\frac{1}{N} , \mbox{ with probability } \frac{1-\eta}{N} , \\
	-\frac{1}{N} , \quad \mbox{ otherwise }  
	\end{cases} \quad \mbox{ for } X_i > 0.
	\end{split} 
	\end{equation}
Note that, while the Shifted IDLA process $(A^*(t))_{t\geq 0}$ is in $\cB \setminus \cA$, the walk $X$ away from~$0$  stochastically dominates the associated excess height process by Lemma \ref{le:towards_height}. In particular, let $N_0$ denote the total number of visits of $X$ to $0$ before reaching $[\cE^* , \infty )$. Then 
	\begin{equation}\label{tauN}
	  \t^i_{\cA , \partial \cB} \succeq N_0 N ,
	 \end{equation}
since each time the walk reaches $0$ it jumps deterministically to $1-1/N$, after which it takes at least $N$ steps to reach $0$ again.
%
Now, $N_0$ is a geometric random variable with success probability $\bbP_{1-1/N} ( X \mbox{ reaches } [\cE^* , \infty ) \mbox{ before } 0 )$. The next result tells us that this probability is very small, and so $N_0$ is typically large. 
\begin{Lemma}\label{postpone}
For $\cE^*$ as in \eqref{eq:E*} and $N$ large enough, it holds 
 	\begin{equation} \label{eq:postpone}
 	\bbP_{1-1/N} ( X \mbox{ reaches } [\cE^* , \infty )  \mbox{ before } 0 ) \leq  e^{-N} \,.
 	 \end{equation} 
\end{Lemma}
Let us postpone the proof of the above lemma to Appendix \ref{app:postpone}, and explain how from this one can deduce the bound of Lemma \ref{le:Kt}. 
Note that \eqref{tauN} and \eqref{eq:postpone} together imply that 
	\[ \bbE ( \t^i_{\cA , \partial \cB} ) \geq  N e^{N} \]
for $N$ large enough. 
Let now  $T^X_{\cE^*} := \inf \{ i\geq 0 : X_i \geq  \cE^* \}$ denote the first hitting time of $[\cE^* , \infty )$ for the walk $X$, and notice that $T_{\cE^*}^X \geq N_0 N $. 	 We define  $\big( T^{X,(i)}_{\cE^*}\big)_{i\geq 1}$ to be i.i.d.\ copies of $T^X_{\cE^*}$, independent of everything else. 
Recall the definition of $\tilde{K} (t)$, and define further 
\[ \tilde{K}'(t) := \sup \Big\{ k\geq 1 : \sum_{i=1}^k   T^{X,(i)}_{\cE^*}  \leq t \Big\} .\]
Then we have 
	\begin{equation}\label{Kg}
	 \begin{split} 
	\bbP (\tilde{K}(t) \geq  \g t ) & \leq \bbP ( \tilde{K}'(t) \geq  \g t ) = 
	\bbP \bigg( \sum_{i=1}^{\g t} T^{X,(i)}_{\cE^*}  \leq t \bigg) .
	\end{split} 
	\end{equation}
\begin{Lemma}\label{post2}
Let $\g = \frac{6}{N e^{N}  } $. Then for $N$ large enough it holds 
	\[ \bbP \bigg( \sum_{i=1}^{\g t} T^{X,(i)}_{\cE^*}  \leq t \bigg)  
	\leq \exp \Big( - \frac{t}{N e^{N} }\Big) . \] 
\end{Lemma}
This lemma is an easy consequence of the fact that $T^X_{\cE^*} \geq N_0 N$ and Lemma \ref{postpone}, so we leave the proof for Appendix \ref{app:post2}. This finishes the proof of Lemma \ref{le:Kt}. 
\end{proof}
\vspace{1mm}
Using \eqref{gamma} we can conclude the proof of  Proposition \ref{pr:initial_excess}. Indeed, writing $\cE (t) $ in place of $\cE (A^*(t))$ for brevity, and taking $\g $ as above,  we find
\[ \begin{split} 
	\mu_N (\cB^c ) &= \lim_{t\to\infty } \frac{1}{t} \bbE \bigg( \sum_{i=1}^t \mathds{1} ( \cE (t) \geq 2\cE^* ) \bigg) 
	\leq \lim_{t\to\infty } \frac{1}{t} \bbE \bigg( t- \sum_{i=1}^{K(t)} \t_{\cA ,\partial \cB }^{i}  \bigg) 
	\\ & \leq  \lim_{t\to\infty } \bigg[ \frac{1}{t}  \bbE \bigg( t- \sum_{i=1}^{K(t)} \t_{\cA ,\partial \cB }^{i}  \, ;  \, K(t) < \g t \bigg)  + \bbP (K(t) \geq  \g t ) \bigg] 
	\\ & \leq  \lim_{t\to\infty } \bigg[ \frac{1}{t} \bbE \bigg( \sum_{i=1}^{\g t} \t_{\partial \cB  , \cA }^{i} \bigg) + \exp \Big( - \frac{t}{N e^{N} }\Big) \bigg] 
	= \g \bbE ( \t_{\partial \cB ,\cA }^{1} ) 
	= \frac{ 2 \bbE ( \t_{\partial \cB ,\cA }^{1} )}{N e^{N} }  . 
	\end{split} \]
Lemma \ref{le:excess_height} with $t_0 := \frac{N}{\eta} \cE^* $ then yields 
	\[ \begin{split} 
	\bbE ( \t_{\partial \cB ,\cA }^{1} ) & 
	=\int_0^\infty \bbP ( \t_{\partial \cB ,\cA }^1 >t ) dt 
	\leq t_0  + \int_{t_0 }^\infty \exp \Big( - \frac{\eta^2 t}{8N^2} \Big)  dt 
	\\ & = \frac{N \cE^*}{\eta }  + \frac{8N^2}{\eta^2} \exp\Big( -\frac{\eta  \cE^*}{  8N} \Big)  
	\leq \frac{2N \cE^*}{\eta } 
	\end{split} \]
for $N$ large enough, since $\cE^* > (N \log N)^2$ by \eqref{e*}. 
Thus we conclude that 
	\[ \mu_N (\cB^c ) \leq \frac{4  \cE^* }{\eta} e^{-N} 
	\leq e^{-N/2}  \]
for $N$ large enough, as claimed. 
\end{proof}

\subsection{From high density to low height}
We have shown in the previous section that typical clusters are dense.
While this does not give any information on the height of $A$, it provides an upper bound on the number of empty sites, that we will call \emph{holes}, below the top level $h(A)$.  
Indeed, if $\cE (A) \leq \cE^*$ then  there can be at most $N\cE^*$  holes in the cluster.  We now obtain a bound on the time it takes to fill these holes (cf.\ Proposition \ref{pr:discrepancy}),  showing that it is at most polynomial in $N$, and use this to prove that stationary clusters have at most polynomial height (cf.\ Proposition \ref{pr:poly_height}).

\begin{Proposition}\label{pr:discrepancy}
Let $(A^*(t))_{t\geq 0}$ denote a Shifted IDLA process on $\bbZ_N \times \bbZ$, and assume that  $\cE(A^*(0)) \leq \cE^*$, with $\cE^*$ as in \eqref{eq:E*}.   
Let 
	\begin{equation} \label{eq:Dnew}
	\D := \left\lceil 2N^2 \cE^*\right\rceil +1 . 
	\end{equation} 
Write  $h^*(t)$ in place of $ h(A^*(t)) $ for brevity, and assume that $h^*(0) > \D$. 
Define  $T_\D := \inf \{ t\geq 0 : h^*(t) \leq \D \} $ to be the first time the height of the Shifted IDLA process drops below $\D$. Then there exists a constant $C_\eta >0$, depending only on $\eta$, such that 
	\begin{equation} \label{eq:discrepancy} 
	\bbP ( T_\D > t ) \leq 4 \exp \Big( - \frac{ C_\eta t}{ N\D  } \Big)  
	\end{equation}
for $N$ large enough. 
\end{Proposition}
\begin{Remark}
Since $N\D$ is polynomial in $N$, this tells us that, when starting from a dense configuration,  the height of a Shifted IDLA process drops below $\D$ after at most polynomially many releases. 
\end{Remark}


\begin{proof}
We argue as follows. Each time we add a new particle to the cluster, the lowest hole has probability at least $1/N$ to be filled, independently of everything else. Hence it will take at most a geometric number of releases of parameter $1/N$ to fill the lowest hole. In total, then, it will take at most the sum of $N\cE^*$ i.i.d.\ Geometric$(1/N)$ to fill all the holes up to the top level  $h(A)$. Some care is needed, though: we cannot let the excess height increase too much while releasing these extra particles. 

Let $(G_i)_{i=1}^{N\cE^*} $ be a collection of i.i.d.\ Geometric$(1/N)$ random variables, and note that since $\D > 2 N^2 \cE^* $ we have
	\begin{equation}\label{eq:D}
	 \bbP \bigg( \sum_{i=1}^{N\cE^*} G_i > \D \bigg) < \frac{1}{2} \, .
	 \end{equation}
It follows that if we release $\D$ particles then we have probability at least $1/2$ to fill all the holes below the top level. If we fail, then the excess height has increased by at most $\D$. If so, we keep on releasing particles until the excess height falls again below $\cE^*$, and iterate. 
After a Geometric$(1/2)$ number of attempts we will have filled all the holes below the top level, which implies that the resulting Shifted IDLA cluster will have height at most $\D$. 

To formalise the above strategy, write $\cE (t)$ in place of $\cE (A^*(t))$ and define the following random times: 
	\begin{align*}
	& \t_0 := 0 ,  \\
	& \t_k :=  \inf \{ t\geq \t_{k-1}+\D : \cE (t) \leq  \cE^* \}  \, , \quad k\geq 1, \\
	 & s_k := \t_k -( \t_{k-1} +\D ) \, , \quad k\geq 1.
	\end{align*}
Then, by definition, at times $\t_i$ there are at most $N\cE^*$ holes below the top level. Since the number of releases needed to fill all these holes is stochastically dominated by the sum of~$N\cE^*$ i.i.d.\ Geometric$(1/N )$ random variables, \eqref{eq:D} implies that at times $\t_i +\D$ we have filled all the holes with probability at least $1/2$. If this happens, then $h^*(\t_i +\D) \leq \D$, and we stop. Otherwise $\cE (\t_i +\D) \leq \cE^* +\D$, and we start again. 

Let 
	\[ K := \min \{ k\geq 0 : h^* (\t_k +\D) \leq \D \} \]
denote the number of attempts needed to succeed. Then, denoting by $\preceq $  stochastic domination, we have  $K \preceq K'$ for $K'$  Geometric$(1/2)$ random variable.  
Note that the times $s_i$'s are not in general identically distributed. 
It is convenient to make them i.i.d.\ by assuming that the excess height always starts from the maximum value $\cE^*+\D$. More precisely, let 
	\begin{equation} \label{hatS}
	 \hat s:= \inf \{ t\geq 0 : \cE(t) \leq \cE^* \mbox{ starting from } \cE(0) =\cE^*+\D \} ,
	 \end{equation}
and let $(\hat s_i)_{i\geq 1}$ be a sequence of i.i.d.\ random variables equal in law to $\hat s$. Then $s_i \preceq \hat{s}_i$ and 
we have that 
	\[ T_\D \preceq  K\D + \sum_{i=1}^K \hat s_i 
	\preceq  K'\D + \sum_{i=1}^{K'} \hat s_i . \]
We use this to estimate the moment generating function of $T_\D$. 
For $\l \in \bbR$, set~$M (\l) := \bbE ( e^{\l \hat{s}} )$. We have
	\begin{equation} \label{eq:chain}
	 \begin{split} 
	\bbE (e^{\l T_\D } ) & \leq \bbE \bigg[ \exp\Big\{\l \Big( K'\D + \sum_{i=1}^{K'} \hat s_i \Big)\Big\} \bigg] \\ & 
	= \sum_{k=1}^\infty \bbE \bigg[ \exp\Big\{\l \Big( k\D + \sum_{i=1}^k s_i \Big)\Big\}  \mathds{1}(K'=k)\bigg] \\ 
	& \leq \sum_{k=1}^\infty \bigg( \frac{e^{\l \D} M(\l)}{2} \bigg)^k \,,
	\end{split}
	\end{equation}
where the first equality above follows from the Monotone Convergence Theorem, and the last inequality from 
the independence of $K'$ and the $\hat{s}_i$'s.
\begin{Lemma}\label{le:Ml}
For $\l < \frac{\eta^2}{32 N^2}$ and $N$ large enough, it holds $M(\l ) \leq \frac{3}{2} \exp \big( \frac{2\D N \l }{\eta} \big)$. 
\end{Lemma}
The above lemma implies that $ \frac{e^{\l \D} M(\l ) }{2} 
	\leq \frac{3}{4} \exp \big\{ \l \D \big( 1 + \frac{2N}{\eta} \big) \big\} $, 
which can be made smaller than, say, $4/5$ by taking 
	\[ \l < \frac{ \log (16/15)}{\D ( 1+\frac{2N}{\eta} )} . \]
Let  $C_\eta = \frac{\log (16/15)}{4}\eta$, so that $\frac{C_\eta}{N\D} <  \frac{ \log (16/15)}{\D ( 1+\frac{2N}{\eta} )}$ for large $N$. If $\l^* = \frac{C_{\eta}}{N\D }$ then  we have 
	\[ \bbE ( e^{\l^* T_\D} ) \leq \sum_{k=1}^\infty \Big( \frac{4}{5} \Big)^k 
	= 4 . \]
Thus we conclude that 
	\[ \bbP ( T_\D > t ) \leq \bbE ( e^{\l^* T_\D} ) e^{-\l^* t} 
	\leq 4 e^{-\l^* t} = 4 \exp \Big( - \frac{ C_\eta t }{N\D } \Big) , \]
as claimed. It remains to prove Lemma \ref{le:Ml}.
\begin{proof}[Proof of Lemma \ref{le:Ml}]
Recall the definition of $\hat{s}$ from \eqref{hatS}, so that $M(\l ) = \bbE (e^{\l \hat{s}} )$. Introduce the auxiliary random walk $(X_k)_{k\geq 0}$ defined as in \eqref{rwX}, and note that 
	\[ \hat{s} \preceq  \hat{\t}_\D 
	:= \inf\{ t\geq 0 : X_t \leq \cE^* \mbox{ starting from } \cE^* +\D \} . \]
Then it follows from Lemma \ref{le:excess_height} that, for $t\geq \frac{2\D N}{\eta}$ and $N$ large enough, it holds 
	\[\bbP ( \hat{\t}_\D > t ) \leq \exp \Big( -\frac{\eta^2 t}{8N^2} \Big). \]
Therefore, for $\l < \frac{\eta^2}{32 N^2}$, we find 
	\[ \begin{split}
	\bbE ( e^{\l \hat{s}} ) & \leq \bbE (e^{\l \hat{\t}_\D }  ) 
	= \int_0^\infty \bbP ( e^{\l \hat{\t}_\D} \geq t ) dt 
	= \int_0^\infty \bbP \Big( \hat{\t}_\D \geq \frac{\log t}{\l}  \Big) dt 
	\\ & 
	\leq \exp \Big( \frac{2\D N \l}{\eta} \Big)  + \int_{\exp \big( \frac{2\D N \l }{\eta} \big) }^\infty \exp \Big( -\frac{\eta^2 \log t }{8N^2 \l } \Big) dt 
	\\ & = \exp \Big( \frac{2\D N \l}{\eta} \Big) + 
	\frac{ \exp \big( - \frac{2\D N \l }{\eta} \big( \frac{\eta^2}{8 N^2 \l } -1 \big) \big) }{\frac{\eta^2}{8 N^2 \l} -1 } 
	\\ & \leq \frac{3}{2} \exp \Big( \frac{ 2\D N \l }{\eta } \Big) , 
	\end{split}\]
where the last inequality holds for $N$ large enough. 
\end{proof}
This concludes the proof of Proposition \ref{pr:discrepancy}. 
\end{proof}

We use this to show that stationary clusters have at most polynomial height. 
\begin{Proposition}\label{pr:poly_height}
For $N$ large enough, it holds 
	\begin{equation*}
	 \mu_N \Big( \big\{ A :  h(A) > N^8 \big\} \Big)  
	\leq  2e^{-N/2} .
	\end{equation*}
\end{Proposition}
\begin{proof}
Let $(A(t))_{t\geq 0}$ be an IDLA process with $A(0) = A \sim \mu_N$. Then $A(t) \sim \mu_N$ for all deterministic $t\geq 0$. Write $h(t)$ in place of $h(A(t))$ for brevity. Take $\cE^* = N^2 \log^3 N$, and note that it satisfies \eqref{eq:E*}. Define $\D$ as in \eqref{eq:Dnew}. Then we have 
	\[ \begin{split} 
	\bbP ( h(A) > N^8 ) & = \bbP (h(N^7) > N^8 | A(0) =A) \\ & 
	\leq \bbP ( h(N^7) > N^8 | \cE (A(0)) \leq 2\cE^* ) + \bbP (\cE (A(0)) > 2\cE^* )
	\\ & \leq \bbP (T_\D > N^7 | \cE (A(0)) \leq 2\cE^* ) + e^{-N/2} 
	\leq 2e^{-N/2} ,
	\end{split} \]
where we have used that if $T_\D \leq N^7$ then $h(N^7) \leq \D + N^7 < N^8$. 
\end{proof}

\begin{Remark}
It is worth pointing out that the proof of Proposition \ref{pr:initial_excess}, and hence of Proposition \ref{pr:poly_height} above, would also work for driving random walks on $\bbZ_N \times \bbZ $ with a vertical drift. This would still give a polynomial bound for the height of typical clusters, perhaps with a larger exponent than the one in Proposition \ref{pr:poly_height}. Reasoning as in Theorem \ref{th:main2}, such height bound could in turn be translated into a (largely non-optimal) upper bound for the time it takes for IDLA with transient driving walks to forget its initial profile. 
\end{Remark}

\subsection{Typical clusters are shallow} \label{sec:shallow2}
We now finish the proof of Theorem \ref{pr:typical_height}. 
To start with, note that it suffices to prove the result for stationary clusters. Indeed, suppose we showed that for any $\g >0$ there exists a constant $c_\g$ such that 
	\[ \mu_N \left( \{ A : h(A) > c_\g \log N \} \right) \leq N^{-\g} \]
for $N$ large enough. Then, if $\nu_N$ is any $k$-lukewarm start for Shifted IDLA, we have 
	\[  \nu_N \left( \{ A : h(A) > c_\g \log N \} \right) \leq N^{-(\g -k)} , \]
so that \eqref{numu} still holds with $\g +k$ in place of $\g$. It thus suffices to prove Theorem \ref{pr:typical_height} for stationary clusters. We argue as follows.
Let $(A(t))_{t\geq 0}$ be an IDLA process starting from $A(0) \sim \mu_N$, so that $A(t) \sim \mu_N$ for all deterministic $t\geq 0$. Introduce an auxiliary IDLA process $(\bar{A}(t))_{t\geq 0}$ starting from the flat profile $\bar{A}(0)=R_0$. Then, if $|A(0)|=n_0$, we have 
	\[ |A(t) | = |\bar{A}(n_0+t)| = n_0 +t , \qquad \forall t\geq 0 . \]
Write $A'(t) = \bar{A}(n_0+t)$ to shorten the notation. 
Following the ideas presented in the proof of Theorem \ref{th:main2}, we will couple the clusters $A(N^{10})$ and $A'(N^{10})$ so that they match with high probability. Since $A(N^{10}) \sim \mu_N$, and $A'(N^{10})$ has logarithmic fluctuations with high probability, this will allow us to conclude.

To start with, note that by Proposition \ref{pr:poly_height} we have 
	\[ \bbP \big( h(A(0)) > N^8 \big) \leq 2e^{-N/2} \]
for $N$ large enough. In particular this shows that $\bbP (n_0 >N^9) \leq 2e^{-N/2}$. 
We can use this to bound the height of $A'(0)$. Indeed, for $\g$ as in the statement of Theorem~\ref{pr:typical_height}, we have 
	\[ \begin{split} 
	 \bbP ( h(A'(0)) > 2N^8 ) & \leq 
	\bbP ( h(\bar{A}(n_0)) > 2N^8  , n_0 \leq N^9) + \bbP ( n_0 > N^9 ) 
	\\ & \leq \bbP \big( h(\bar{A} (N^9) ) > 2N^8 \big) + 2e^{-N/2} 
	\leq N^{-2\g} + 2e^{-N/2} \leq 2N^{-2\g} 
	\end{split} \] 
for $N$ large enough, where in the last inequality we have used Theorem \ref{pr:height_bound} to argue that an IDLA cluster built by adding $N^9$ particles to $R_0$ has height $\cO (N^8)$ with high probability. 
Introduce a water cluster $W_0$ obtained by adding $N^{10}$ particles to the flat configuration $R_0$ according to IDLA rules, and note that
	\begin{equation} \label{eq:Wfilled}
	 \bbP \big( W_0 \supseteq R_{3N^8} \big) \geq 1- N^{-2\g} 
	 \end{equation} 
by Theorem \ref{pr:height_bound} for $N$ large enough. 
Define two auxiliary water processes $(W(t))_{t\geq 0}$ and $(W'(t))_{t\geq 0}$ as follows. Set $W(0) = W'(0) = W_0$. Particles in $W(0) \cap A(0)$ and $W'(0) \cap A'(0)$ are declared frozen, and will be released at a later time. Note that 
	\begin{equation} \label{eq:event}
	\bbP \Big( \max\{ h(A(0)) , h(A'(0)) \} \leq 2N^8 ; \, W_0 \supseteq R_{3N^8} \Big) 
	\geq 1-4N^{-2\g} ,
	\end{equation} 
which shows that all frozen particles are at distance at least $N^8$ from the boundary of $W_0$ with high probability. 
Fix arbitrary enumerations of the two sets of frozen particles, and accordingly denote their locations by $\{ z_1 , z_2 , \ldots , z_{n_0}\}$ and $\{ z'_1 , z'_2 , \ldots , z'_{n_0}\}$. 
For $t\geq 0$, then, let $W(t)$ (respectively $W'(t)$) be the cluster obtained by adding to $W(t-1)$ (respectively $W'(t-1)$) the exit location from it of a simple random walk on $\bbZ_N \times \bbZ$ starting from $z_t$ (respectively~$z'_t$). The random walks starting from $z_t$ and $z'_t$ are coupled as explained in the introduction: the higher one stays in place until the other one reaches its level, after which they move together in the vertical coordinate, and according to the reflection coupling in the horizontal one\footnote{If needed, we use the first horizontal step to adjust the parity of the difference of the horizontal coordinates, as explained in the proof of Proposition \ref{pr:coupling}.}. 
Then, writing $E_0$ for the event appearing in \eqref{eq:event} for brevity, by Proposition \ref{pr:coupling} we find 
	\[ \begin{split} 
	\bbP ( W(n_0) \neq  W'(n_0) ) & \leq  \bbP (  W(n_0) \neq  W'(n_0)  | \,  E_0 ) + \bbP(E_0^c) 
	\\ & \leq N^{-2\g} + 4N^{-2\g} = 5N^{-2\g} 
	\end{split} \]
for $N$ large enough. 
Thus  $W(n_0) = W'(n_0)$ with high probability. Since $A(N^{10})\stackrel{(d)}{=} W(n_0)$ and $A'(N^{10}) \stackrel{(d)}{=} W'(n_0)$, this shows that we can couple $A(N^{10})$ and $A'(N^{10})$ so that 
	\[ \bbP \big( A(N^{10}) = A'(N^{10})\big) \geq 1-5N^{-2\g} . \]
Now, $A(N^{10})$ is stationary since $A(0)$ is. We want to argue that $A'(N^{10}) = \bar{A}(n_0 + N^{10} ) $ has logarithmic fluctuations. If $n_0$ were deterministic, this would follow from Theorem \ref{pr:height_bound}. Instead $n_0$ is random since such is $A(0)$, and as already observed it satisfies \linebreak $\bbP (n_0 > N^9 )  \leq 2e^{-N/2} $ for $N$ large enough. 
Moreover, by Theorem \ref{pr:height_bound} there exists a constant $b_{\g }$, depending only on $\g$, such that 
	\[ 	 \bbP \Big(  R_{\frac{t}{N} - b_{\g }  \log N }
	\subseteq  A(t) \subseteq
	R_{\frac{t}{N} + b_{\g }  \log N  }  \; \; \forall t \leq 2N^{10} \Big) \geq 1-N^{-2\g} \]
for $N$ large enough. We thus conclude that 
	\[ \begin{split} 
	\bbP \Big( \Big\{  R_{\frac{n_0+N^{10}}{N} - b_{\g }  \log N }  &
	\subseteq  A'(N^{10}) \subseteq
	R_{\frac{n_0+N^{10}}{N} + b_{\g }  \log N  }  \Big\}^c \Big) =
	\\ & = \bbP \Big( \Big\{  R_{\frac{n_0+N^{10}}{N} - b_{\g }  \log N }  
	\subseteq  \bar{A}(n_0+N^{10}) \subseteq
	R_{\frac{n_0+N^{10}}{N} + b_{\g }  \log N  }  \Big\}^c \Big)  \leq 
	\\ & \leq \bbP \Big( \Big\{ R_{\frac{t}{N} - b_{\g }  \log N }  
	\subseteq  \bar{A}(t) \subseteq
	R_{\frac{t}{N} + b_{\g }  \log N  }   \;\; \forall t \leq 2N^{10} \Big\}^c \Big) + 
	\bbP (n_0 > N^9) 
	\\ & \leq N^{-2\g} + 2e^{-N/2} \leq 2N^{-2\g} 
	\end{split} \]
for $N$ large enough. This shows that $A'(N^{10})$ has logarithmic fluctuations with high probability. Recall that for $A \in \Omega$ we denote by $A^*$ its shifted version (cf.\ Definition \ref{def:shifted}). Then we have found that 
	\[ \begin{split} 
	\mu_N \Big( \big\{ A : h(A) > 2b_{\g} \log N \big\} \Big) 
	& = \bbP \big( h(A^*(N^{10})) > 2 b_{\g} \log N \big) 
	\\ & \leq \bbP \big( h(A'^*(N^{10})) > 2 b_{\g} \log N \big)  + 
	\bbP \big( A(N^{10}) \neq A'(N^{10}) \big) 
	\\ & \leq 7 N^{-{2\g} } \leq N^{-\g}
	\end{split}
	\]
for $N$ large enough, which concludes the proof.

\section{The upper bound} \label{sec:proof_upper_bound}
We briefly explain how one can deduce Theorem~\ref{th:ZNmain_teo} from Theorems~\ref{th:main2} and~\ref{pr:typical_height}. For any $\g , k>0$ let $c_{\g , k}$ be defined as in Theorem \ref{pr:typical_height}, while we take $d'_{\g , 2}$ as in Theorem \ref{th:main2} with $m=2$. Set 
	\[ d_{\g , k} := c_{ \g ,k} + d'_{\g , 2} , \qquad 
	t_{\g , k} := d_{\g ,k} N^2 \log N . \]
Then, if $\nu_N$ is any 	$k$-lukewarm distribution, we can define 
	\[ \Omega_{\g ,k} := \{ A : h(A) \leq c_{\g ,k} \log N \} \]
to have that, by Theorem~\ref{pr:typical_height},
	\[ \bbP \big( \Omega_{\g ,k} \big) \geq 1-N^{-\g} \]
for $N$ large enough. 
Take any two clusters $A_0$, $A'_0 $ in $\Omega_{\g, k}$ with $|A_0|=|A'_0|$, and let $(A(t))_{t\geq 0}$ and $(A'(t))_{t\geq 0}$ be two IDLA processes starting from $A_0$ and $A'_0$ respectively. Then Theorem \ref{th:main2} tells us that we can build $A(t_{\g , k})$ and $A'(t_{\g , k})$ on the same probability space so that 
	\[ \bbP \big( A(t_{\g , k}) \neq A'(t_{\g , k})  \big) \leq N^{-\g}  \]
for $N$ large enough.  
We thus gather that 
	\[ \begin{split} 
	 \| P(t_{\g , k}) - P' (t_{\g , k}  )\|_{TV}  & = 
	\sup_{S\subseteq \Omega } \big| \bbP (A(t_{\g , k} ) \in S ) - \bbP (A'(t_{\g , k} )\in S ) 
	\big| \\ & \leq 2 \bbP \big( A(t_{\g , k}) \neq A'(t_{\g , k}) \big) 
	\leq 2N^{-\g} 
	\end{split} \]
for $N$ large enough. Since $\g$ is arbitrary, this concludes the proof.

\section{The lower bound} \label{sec:lower_bound} 
In this section we specialise to stationary initial clusters, and prove that if two such clusters are sampled \emph{independently} from $\mu_N$, then the IDLA dynamics will remember from which one it started for at least order $N^2$ steps, as stated in Theorem \ref{th:ZNlower_bound}. 
We proceed as follows. We first recall the GFF fluctuations result by Jerison, Levine and Sheffield \cite{jerison2014internal2}, saying that the average IDLA fluctuations, appropriately measured, converge to the restriction of the Gaussian Free Field to the unit circle. We then use this to define an observable which we show to be large, with positive probability, for stationary IDLA clusters (cf.\ Proposition \ref{initial_imbalance}). Finally, we argue that a necessary condition for IDLA to forget the initial configuration is for this observable to reach $0$, and show that this takes time at least $\a N^2$ for some $\a >0$, thus proving the result.  

\subsection{Average IDLA fluctuations and the GFF}
Let us start by briefly recalling the average IDLA fluctuations result by Jerison, Levine and Sheffield \cite{jerison2014internal2}.  
Let $(A(t))_{t\geq 0}$ be an IDLA process on $\bbZ_N \times \bbZ$ starting from flat, i.e.\ $A(0) = R_0$. 
For $n=(n_1 , n_2) \in \bbZ_N \times \bbZ$, define the rescaled square 
	\[ Q_N(n) := \Big\{ (x,y) \in \bbT \times \bbR : 
	x \in \Big( \frac{n_1 -1}{N} , \frac{n_1}{N} \Big] , \; 
	y \in \Big( \frac{n_2 -1}{N} , \frac{n_2}{N} \Big] \Big\} \]
with side--length $1/N$ and $\big( \frac{n_1}{N} , \frac{n_2}{N} \big)$ as top--right corner, and set 
	\begin{equation} \label{rescale_fill}
	 A_N(t) := \bigcup_{n \in A(t)} Q_N(n) , 
	 \end{equation}
so that $A_N(t) \subset \bbT \times \bbR$. We use the rescaled and filled cluster $A_N(t)$ to define the \emph{discrepancy function} 
	\begin{equation} \label{DT}
	 D_{N,t} (x,y) := N \Big( \one_{A_N(t)} (x,y) - \one_{ \big\{ y\leq \frac{t}{N^2} \big\} } (x,y) \Big)  
	 \end{equation}
for $(x,y) \in \bbT \times \bbR$. Note that $D_{N,t}$ is supported on the symmetric difference $A_N(t) \Delta R_{t/N^2}$. 
Finally, let $\varphi \in C^{\infty} (\bbT \times \bbR )$ be of the form 
	\[ \varphi (x,y) = \sum_{|k|\leq K} \a_k(y) e^{2\pi i kx} \]
for some finite integer $K$ and with $\a_{-k} = \overline{\a_k}$, so that $\varphi$ is real--valued. Jerison, Levine and Sheffield proved the following. 

\begin{Theorem}[Theorem 3, \cite{jerison2014internal2}] \label{JLS_average}
Let $T=y_0 N^2$  and $\varphi$ be as above. Then, as $N \to \infty$, 
	\[
	 D_{N,T}(\varphi ) := \int_{\bbT \times \bbR } D_{N,T}(x,y) \varphi(x,y) \dd x \dd y 
	 \]
converges in distribution to a Gaussian random variable with mean zero and variance 
	\[ v(\varphi ) = \sum_{0<|k|\leq K} |\a_k(y_0)|^2 
	\Big( \frac{ 1-e^{-4 \pi |k| y_0}}{4\pi |k|} \Big) . \]   
\end{Theorem} 

The next result tells us that the above theorem can be generalised to larger times. 
\begin{Theorem} \label{th:GFF1}
Let $T=CN^2 \log N$ for some absolute constant $C$ large enough. Then, 
as $N\to\infty$, 
	\[ \int_{\bbT \times \bbR} D_{N,T}(x,y) \varphi \Big( x,y - \frac{T}{N^2} \Big) \dd x \dd y \]
converges in distribution to a Gaussian random variable with mean zero and variance
	\[ v_\mu (\varphi ) = \sum_{0<|k|\leq K} \frac{|\a_k(0)|^2}{4\pi |k|}  . \]   
\end{Theorem}
\begin{Remark}
Note that, as observed in \cite{jerison2014internal2}, the exponential term in $v(\varphi )$ is due to the fact that the process started from flat. Indeed, this term does not appear in the limiting variance $v_\mu (\varphi )$, since $T$ is large enough for the process to have reached stationarity. 
\end{Remark} 
The above result can be proved exactly as in \cite{jerison2014internal2}, Theorem 3, by replacing the maximal fluctuations result with Theorem \ref{pr:height_bound} above. For this reason we choose to skip the proof.  

We now use the fact that $A(T)$ is close to a stationary cluster to control the size of the average fluctuations at stationarity. Let $A^\mu$ denote a stationary cluster, that is $A^\mu \sim \mu_N$, and denote by $A_N^\mu$ its rescaled and filled version as in \eqref{rescale_fill}. We start an IDLA process $(A^\mu (t))_{t\geq 0}$ from $A^\mu (0) = A^\mu$. Let $n_0 = |A^\mu (0) |$ be the number of particles in $A^\mu (0)$ above level zero. Then, if $(A(t))_{t\geq 0}$ is an IDLA process starting from flat, we have 
	\[ |A(n_0)| = |A^\mu (0) |=n_0  \]
and hence $|A(n_0 +t)| = |A^\mu (t)|=n_0+t$ for all $t\geq 0$. 
Moreover, by Theorems~\ref{pr:height_bound} and~\ref{pr:typical_height}, for any $\g >0$ there exists $a_\g <\infty$ such that, with $h_0 =  \max\{ h(A(n_0)) , h(A^\mu (0) ) \} $, it holds 
	\[ \bbP \big( h_0 > a_\g \log N \big) \leq 2N^{-\g} \]
for $N$ large enough. Consider the time--shifted IDLA process $A'(t) = A(n_0 +t)$ starting from logarithmic height with high probability. 
Then by Theorem \ref{th:main2} for any $\g >0$ there exists a finite constant $d_\g$ such that for $t_\g = d_\g N^2 \log N$ we can couple the clusters $A^\mu (t_\g ) , A'(t_\g)$ so that 
	\[ \bbP \big( A^\mu (t_\g ) \neq A'(t_\g) , \, h_0 \leq a_\g \log N \big) \leq 3N^{-\g}  \]
for $N$ large enough. 
Let $T= t_\g + n_0$. Recall from \eqref{DT} the definition of $D_{N,T}$, and define 
	\begin{equation} \label{Dmu}
	 D_{N,t_\g}^\mu (x,y) :=  N \Big( \one_{A_N^\mu (t_\g) } (x,y) - 
	 \one_{ \big\{ y\leq \frac{T}{N^2} \big\} } (x,y) \Big) . 
	\end{equation}
For $\varphi$ as above, introduce the random variables 
	\[ \begin{split} 
	X_N^\varphi & : = \int_{\bbT \times \bbR } D_{N,T}(x,y) 
	\varphi \Big( x,y-\frac{T}{N^2} \Big) \dd x \dd y , \\
	X_N^{\mu , \varphi} & :=  \int_{\bbT \times \bbR } D^{\mu} _{N,t_\g}(x,y) 
	\varphi \Big( x,y-\frac{T}{N^2} \Big) \dd x \dd y . 
	\end{split} \]
Then for any $\d >0$ we have 
	\[ \bbP ( X_N^{\mu , \varphi} >\d ) \geq \bbP ( X_N^\varphi >\d , A^\mu (t_\g ) = A (T)) 
	\geq  \bbP ( X_N^\varphi >\d ) - 5N^{-\g} . \]
Moreover, by Theorem \ref{th:GFF1} for any $\e >0$ we can take $N$ large enough to ensure that, if $\cN$ is a standard Gaussian random variable,  
	\[ \bbP (  	X_N^\varphi >\d ) \geq  \bbP \Big(\cN > \frac{\d }{\sqrt{v(\varphi )}} \Big) -\e 
	\geq \frac{1}{2} - \Big( \frac{\d }{\sqrt{2\pi v_\mu (\varphi )}} +\e \Big) ,  \]
from which 
	\begin{equation}\label{Xmu}
	 \bbP ( X_N^{\mu , \varphi} >\d )  \geq 
	\frac{1}{2} - \Big( \frac{\d }{\sqrt{2\pi v_\mu (\varphi )}} +\e + 5N^{-\g} \Big) , 
	\end{equation}
for any $\g ,\d , \e >0$ and $N$ large enough.


\subsection{The initial imbalance}
We now make a choice for the test function $\varphi$ in the above discussion. Let 
	\[ \phi (x,y) := \sin(2\pi x ) = \frac{e^{2\pi i x } - e^{-2\pi i x }}{2} , \]
for which $v_\mu (\phi ) = \frac{1}{8\pi}$. Then by \eqref{Xmu} we can take $\e = \d$ and $N$ large enough to get 
	\begin{equation} \label{choose_phi}
	 \bbP ( X_N^{\mu , \phi} >\d )  \geq 
	\frac{1}{2} - 4\d  .  
	\end{equation}
On the other hand, 
	\[\begin{split} 
	 X_N^{\mu , \phi} & 
	 = \int_{ \bbT \times \bbR} D_{N,t_\g}^\mu (x,y) \phi (x,y) \dd x \dd y 
	= N \int_{\bbT \times \bbR  } \phi (x,y) \one_{A_N^\mu (t_\g)} (x,y) \dd x \dd y 
	\\ & = N \int_{A^\mu_N(t_\g)} \phi (x,y) \dd x \dd y 
	\stackrel{(d)}{=} N \int_{A^\mu_N} \phi (x,y) \dd x \dd y , 
	\end{split}\] 
where $A^\mu \sim \mu_N$ and the last equality holds in distribution. 
In order to build a martingale, we now approximate $\phi$ by its discrete harmonic extension $\psi$ away from the line $\{ y=0 \}$ on the high probability event 
	\[ E_\mu  = \big\{ h(A^\mu ) \leq c_\g \log N \} , \]
with $c_\g$ as in Theorem~\ref{pr:typical_height} with $k=0$.
Following \cite{jerison2014internal2}, to define $\psi$ we introduce $q_N $ solution of 
	\begin{equation} \label{qN} 
	\cosh (q_N/N) = 2-\cos (2\pi /N) , 
	\end{equation}
so that $q_N = 2\pi + \cO (N^{-2})$, and for $n=(n_1 , n_2) \in \bbZ_N \times \bbZ$ set 
	\begin{equation} \label{psi}
	 \psi (n) := e^{q_N n_2/N} \sin \Big( \frac{2\pi n_1}{N}\Big) . 
	 \end{equation}
It is easy to check that $\psi$ is discrete harmonic on $\bbZ_N \times \bbZ$. Moreover, for $n=(n_1,n_2) \in A^\mu $ and $(x,y)\in Q_N(n)$ we have that on the good event $E_\mu$  
	\[ | \phi (x,y) - \psi (n) | = \Big| \sin (2\pi  x ) - e^{q_N n_2/N} \sin  \Big( \frac{2\pi n_1}{N}\Big) \Big| 
	\leq \frac{12\pi c_\g \log N}{N} , \]
from which 
	\[ N \int_{A^\mu_N} \phi (x,y) \dd x \dd y 
	= \frac{1}{N} \sum_{n\in A^\mu } \psi (n) 
	+ \cO \bigg( \frac{\log^2 N}{N} \bigg) \]
for $N$ large enough.
Combining this with \eqref{choose_phi}, we get 
	\[ \bbP \Big( \frac{1}{N} \sum_{n\in A^\mu } \psi (n) > \frac{\d}{2} \Big) 
	\geq \bbP \Big(  N \int_{A^\mu_N} \phi (x,y) \dd x \dd y >\d  , E_\mu \Big)  
	\geq  \frac{1}{2} - 4\d - \bbP (E_\mu) \geq \frac{1}{2} - 5\d \]
for any $\d >0$ and $N$ large enough. 
The same arguments can be used to prove the reverse inequality, thus obtaining  the following. 
\begin{Proposition} \label{initial_imbalance}
Let $A^\mu \sim \mu_N$ be a stationary IDLA cluster. For any $\d >0$ and $N$ large enough, it holds 
	\[ \bbP \bigg( \frac{1}{N} \sum_{n\in A^\mu} \psi (n) > 
	\d  \bigg) \geq \frac{1}{2} - 10\d , 
	\qquad 
    \bbP \bigg( \frac{1}{N} \sum_{n\in A^\mu} \psi (n) < 
	-\d  \bigg) \geq \frac{1}{2} -10 \d  . \]
\end{Proposition}
Let  $c_2$ be defined as in Theorem \ref{pr:typical_height} with $k=0$. We define
	\begin{equation} \label{omega_precise}
	\begin{split} 
	 \Omega_\d & := \Big\{ A \in \Omega : h(A) \leq c_2 \log N \mbox{ and } 
	 \frac{1}{N} \sum_{n\in A} \psi (n) >\frac{\d}{20} \Big\} \, , 
	\\ 
	 \Omega '_\d & := \Big\{ A \in \Omega : h(A) \leq c_2 \log N \mbox{ and } 
	 \frac{1}{N} \sum_{n\in A} \psi (n) < -\frac{\d}{20} \Big\} 
	 \end{split}
	 \end{equation} 
to have that 
	\[ \begin{split} 
	\mu_N ( \Omega_\d ) & \geq 1- \mu_N ( \{ A : h(A)>c_2 \log N \} ) 
	- \bbP \bigg( \frac{1}{N} \sum_{n\in A^\mu} \psi (n) > 
	\d  \bigg) \\ & \geq \frac{1}{2} - \frac{1}{N^2} - \frac{\d}{2} \geq \frac{1}{2}-\d
	\end{split} \]
for $N$ large enough. Similarly, $\mu_N (\Omega_\d ') \geq \frac{1}{2} -\d$ for $N$ large enough, and thus \eqref{omegas} in Theorem~\ref{th:ZNlower_bound} is satisfied. 

\begin{Remark} \label{rem:u0}
It follows from the above result that if $A$ and $A'$ are two independent samples of $\mu_N$, then  for any $\d >0$ we can take $N$ large enough so that 
	\[ \bbP \bigg( \bigg| \frac{1}{N} \sum_{n\in A} \psi (n)
	- \frac{1}{N} \sum_{n\in A'} \psi (n) \bigg| > 2\d  \bigg) 
	\geq 2\Big( \frac{1}{2} - 10 \d \Big)^2 
	\geq \frac{1}{2} - 20\d 
	,\]
which can be made arbitrarily close to $1/2$ by taking $\d $ small enough. 
\end{Remark}

\subsection{The observable} \label{sec:observable}
We use the above remark to define a convenient observable which, loosely speaking, measures the difference in the horizontal imbalance of two IDLA processes. Take $A_0 \in \Omega_\d$ and $A'_0 \in \Omega '_\d $, and assume $|A_0| = |A'_0|$ without loss of generality, as if not then two IDLA processes starting from $A_0$ and $A'_0$ will never meet. We take $A_0 , A'_0$ as starting configurations of two IDLA processes $(A(t))_{t\geq 0}$ and $(A'(t))_{t\geq 0}$. 
\begin{Definition} [Imbalance] \label{def:imbalance}
For $A \in \Omega$, define the \emph{horizontal imbalance} of $A$ by
	\[ u_A := \frac{1}{N} \sum_{n\in A} \psi (n) 
	= \frac{1}{N}  \sum_{n\in A} e^{q_N n_2/N} \sin \Big( \frac{2\pi n_1}{N} \Big) , \]
with $q_N$ as in \eqref{qN}. 
\end{Definition}
We use this to define an observable $u(t)$ which measures the difference in the imbalance of $A(t)$ and $A'(t)$, namely 
	\[ u(t) := u_{A(t)}  - u_{A'(t)} = \frac{1}{N} \sum_{n\in A(t)} \psi (n) - 
	\frac{1}{N} \sum_{n\in A'(t)} \psi (n) .\]
\begin{Remark}
Since $\psi$ is discrete harmonic on $\bbZ_N \times \bbZ$, we have that $(u(t))_{t\geq 0}$ is a discrete time martingale. 
 \end{Remark} 
By Proposition \ref{initial_imbalance} and Remark \ref{rem:u0} we have 
	\begin{equation} \label{eq:first}
	\bbP \big( |u(0)| > \d  ) \geq \frac{1}{2} - 10 \d  
	\end{equation} 
for any $\d >0$ and $N$ large enough. Clearly $A(t) = A'(t) $ implies $u(t) =0$, so if we define 
	\[ T_0 = \inf\{ t\geq 0 : u(t) =0 \} \]
then for any $\a >0$ we have 
	\begin{equation} \label{eq:T0}
	\begin{split} 
	\bbP ( T_0 \leq \a N^2  , |u(0)|>\d  ) & \leq 
	\bbP \Big( \sup_{t\leq \a N^2 } |u(t) - u(0) | > \d \Big) \\ & 
	\leq \frac{1}{\d^{2}}  \bbE \Big( \sup_{t\leq \a N^2} |u(t) - u(0)|^2 \Big) 
	\leq \frac{c\bbE(Q) }{\d^2}
	\end{split} 
	 \end{equation}
for $c$ absolute constant and 
	\[ Q:= \sum_{t=1}^{\a N^2} \bbE ( |u(t) -u(t-1)|^2 | \cF_{t-1} )  , \]
where the last inequality in \eqref{eq:T0} follows by the Burkholder-Davis-Gundy inequality. Now, if we let $n_t$, $n'_t$ denote the settling locations of the $t^{th}$ walkers in the two IDLA processes, we have 
	\[ Q = \frac{1}{N^2 } \sum_{t=1}^{\a N^2} \bbE \Big( | \psi (n_t) - \psi (n'_t) |^2 
	\Big| \cF_{t-1} \Big) , \]
so 
	\[ \bbE (Q) \leq \bbE \Big( \frac{2}{N^2} \sum_{t=1}^{\a N^2} | \psi (n_t) |^2 \Big) 
	+ \bbE \Big( \frac{2}{N^2} \sum_{t=1}^{\a N^2} | \psi (n'_t) |^2 \Big)  . \]
To estimate the above expectations we have to control the height of the clusters $A(\a N^2 )$ and $A'(\a N^2)$, as the function $\psi$ grows exponentially with them. We explain how to control the first expectation, the second one following from the same arguments. 
Introduce the good event 
	\[ E := \{ h(A(\a N^2 )) \leq 100\a N \} \supseteq 
	\{h(A(\a N^2 )) \leq h(A_0) + 60\a N \} . \]
Then the a priori bound in Lemma \ref{le:a_priori} with $m=60$  gives 
	\[ \bbP (E^c ) \leq   e^{-100\a N} \]
for $N$ large enough. On $E$ we have 
	\[ \frac{1}{N^2} \sum_{t=1}^{\a N^2} | \psi (n_t)|^2 
	\leq \a \Big( \max_{n\in A (\a N^2 )} |\psi (n)|^2 \Big) 
	\leq \a \exp \Big\{ \frac{2 q_N}{N} h(A(\a N^2)) \Big\} 
	\leq C_\a \]
for some constant $C_\a $ depending only on $\a$, and $N$ large enough. 
On $E^c$, on the other hand, we trivially have that $h(A(\a N^2)) \leq h(A_0) + \a N^2  \leq 2\a N^2 $, from which 
	\[ \frac{1}{N^2} \sum_{t=1}^{\a N^2} | \psi (n_t)|^2 
	\leq \a \Big( \max_{n\in A (\a N^2 )} |\psi (n)|^2 \Big) 
	\leq \a \exp \Big\{ \frac{2 q_N}{N} h(A(\a N^2)) \Big\} 
	\leq \alpha \exp \big\{ 32 \a N \big\}  , \]
since $q_N \leq 8$ for $N$ large enough. In all, we have found 
	\[ \begin{split} 
	\bbE \left( \frac{1}{N^2} \sum_{t=1}^{\a N^2} | \psi (n_t)|^2  \right) 
	 & \leq  \bbE \left( \frac{1}{N^2} \sum_{t=1}^{\a N^2} | \psi (n_t)|^2  ; \, E \right)  
	 +\bbE \left( \frac{1}{N^2} \sum_{t=1}^{\a N^2} | \psi (n_t)|^2  ; \, E^c \right) 
	 \\ & \leq C_\a + \a e^{32 \a N } \bbP (E^c) 
	 \leq C_\a + \a e^{-68 \a N} 
	 \leq 2C_\a 
	 \end{split}\]
for $N$ large enough. Similarly one can control the same expectation involving $A'(\a N^2)$. 
Thus 
	\[   \bbP ( T_0 \leq \a N^2  , |u(0)|>\d  ) \leq \frac{8c C_\a }{\d^2 } , \]
where $c$ is the absolute constant in \eqref{eq:T0}. Let $\e$ be as in Theorem \ref{th:ZNlower_bound}. Since $C_\a \to 0$ as $\a \to 0$, we can take $\a$ small enough so that $ \frac{8c C_\a }{\d^2} \leq  \e \big( \frac{1}{2} - 10 \d \big)$, to find 
	\[ 	 \bbP ( T_0 \leq \a N^2  , |u(0)|>\d  ) \leq  \e \Big( \frac{1}{2} - 10 \d \Big) . \]
Finally, putting this together with \eqref{eq:first} we gather that 
	\[  \bbP \big( T_0 \leq \a N^2 \big|  |u(0)|>\d  \big) 
	=  \frac{\bbP ( T_0 \leq \a N^2  , |u(0)|>\d  )}{ \bbP (  |u(0)|>\d  )} \leq  \e , \]
which concludes the proof of Theorem \ref{th:ZNlower_bound}.

\appendix

\section{Proof of Lemma \ref{postpone}} \label{app:postpone}
Recall that we consider a random walk $X$ on $\frac{1}{N} \bbN$, with a reflecting barrier at $0$, with transition probabilities given by 
	\begin{equation}\label{rwXapp}
	 \begin{split}
	& X_0 = 0 , \qquad \bbP \Big( X_{i+1} =1 -\frac{1}{N} | X_{i} =0 \Big) =1 , \\
	& X_{i+1} - X_i = \begin{cases}
	1-\frac{1}{N} , \mbox{ with probability } \frac{1-\eta}{N} , \\
	-\frac{1}{N} , \mbox{ otherwise }  
	\end{cases} \quad \mbox{ for } X_i > 0.
	\end{split} 
	\end{equation}
We aim to show that 
	\[ q:= \bbP_{1-1/N} ( X \mbox{ reaches } [\cE^* ,\infty ) \mbox{ before } 0 ) 
	\leq e^{-N} \]
for $N$ large enough. 
Denote $\bbP_{1-1/N} $ simply by $\bbP$ to shorten the notation, and let $T_{0,\cE^*}^X$  be the first time $X$ reaches either $0$ or $[\cE^* ,\infty )$. Then  we have 
	\[ \begin{split} 
	q  = \bbP \Big( X_{T_{0,\cE^*}^X} \geq  \cE^* \Big) 
	& \leq  \bbP \Big( X_{T^X_{0,\cE^*}} \geq  \cE^* , 
	T^X_{0,\cE^*} \leq  N^3 \log N \Big) + 
	\bbP \Big( T^X_{0,\cE^*} > N^3 \log N \Big) 
	\\ & \leq \bbP \bigg( \sup_{t\leq  N^3 \log N} X_t^{{T^X_{0,\cE^*}}} 
	\geq \cE^* \bigg) 
	+ \bbP \Big( T^X_{0} >  N^3 \log N  \Big) \, , 
	\end{split} \]
where $T^X_0$ is the first time the random walk $X$ reaches $0$, and $X^{T_{0,\cE^*}^X}$ denotes the process $X$ stopped upon reaching $0$ or $[\cE^* , \infty )$. We bound the two  terms above separately. 

For the rightmost term, it is easy to check that $M(t) = X_t + \frac{\eta t}{N}$ is a martingale up to time $T_0^X$. If we thus define $M_0(t) = M(t\wedge T_0^X)$, then $(M_0(t))_{t\geq 0}$ is a martingale for all times, with increments bounded by $1$. It therefore follows from Azuma's inequality that for $t\geq 2N/\eta$, 
	\[ \begin{split} 
	\bbP \big( T_0^X > t \big) & = \bbP ( T_0^X > t , X_t \geq 1/N) 
	= \bbP \Big( T_0^X >t , X_t - X_0 \geq -1+\frac{2}{N} \Big) \\ & 
	= \bbP \Big( T_0^X > t , M(t)-M(0) \geq \frac{\eta t+2}{N} -1 \Big) 
	\leq   \bbP \Big( M_0(t) - M_0(0) > \frac{\eta t}{N} -1 \Big) \\ & 
	\leq \exp \Big( - \frac{ ( \frac{\eta t}{N} -1 )^2}{8t} \Big) \leq 
	\exp \Big( - \frac{\eta^2 t}{32 N^2} \Big) , 
	\end{split} \] 
which can be made smaller than $e^{-2N}$ by choosing  $t \geq N^3\log N$.

For the remaining term, 
again by Azuma we have
	\[ \bbP \bigg( \sup_{t\leq N^3 \log N} X_t^{{T^X_{0,\cE^*}}} 
	\geq \cE^* \bigg) \leq 
	\bbP \bigg( \sup_{t\leq N^3 \log N } M (t \wedge {T^X_{0,\cE^*}} )  
	\geq \cE^* \bigg) \leq 
	\exp \Big( - \frac{ (\cE^* )^2}{8 N^3 \log N } \Big) 
	\leq e^{-2N}  \]
for $N$ large enough. In all, we have found that 
	\[ q  = \bbP \Big( X_{T_{0,\cE^*}^X} \geq  \cE^* \Big)  
	\leq e^{-2N} + e^{-2N}  \leq e^{-N} \]
for $N$ large enough. 
This concludes the proof of Lemma \ref{postpone}. 
\section{Proof of Lemma \ref{post2}} \label{app:post2}
We have, for any $\l >0$, 
	\[ \bbP \bigg( \sum_{i=1}^{\g t} T^{X,(i)}_{\cE^*}  \leq t \bigg)  \leq 
	\bbP \bigg( \exp \Big( - \l \sum_{i=1}^{\g t} T^{X,(i)}_{\cE^*}  \Big) \geq e^{-\l t} \bigg)  
	\leq e^{\l t} \Big[ \bbE \Big( e^{-\l T_{\cE^*}^X} \Big) \Big]^{\g t} . \]
To bound the right hand side above, recall that $T_{\cE^*}^X \geq N_0 N$ for $N_0 \sim $Geometric$(p)$, with $p \leq e^{-N}$. Thus if we let $\hat{N}_0 $ be a Geometric random variable of parameter $\hat{p} := e^{-N}$, then we find 
	\[ \begin{split} 
	\bbE \Big( e^{-\l T_{\cE^*}^X} \Big) & = 
	\int_0^\infty \bbP \big( e^{-\l T_{\cE^*}^X}  \geq t \big) \dd t 
	= \int_0^1 \bbP \Big( T_{\cE^*}^X \leq -\frac{\log t}{\l} \Big) \dd t 
	\\ & \leq \int_0^1 \bbP \Big( \hat{N}_0 \leq -\frac{\log t}{\l N } \Big) \dd t 
	\leq \int_0^1 \Big( 1 - (1-\hat{p} )^{-\frac{\log t }{\l N } } \Big) \dd t 
	\\ & \leq  1- \int_0^1  \exp \Big( \frac{2\log t }{\l N e^{N} } \Big) \dd t 
	= 1- \frac{\l N e^N }{\l N e^N +2 } ,
	\end{split} \]
where for the last inequality we have used that $\log (1-x) \geq -2x $ for $x\in (0,1/2)$. 
Thus 
	\[ \begin{split} 
	\bbP \bigg( \sum_{i=1}^{\g t} T^{X,(i)}_{\cE^*}  \leq t \bigg) & \leq 
	e^{\l t } \bigg( 1- \frac{\l N e^N }{\l N e^N +2 }\bigg)^{\g t} 
	\\ & =  \exp \Big\{ \l t + \g t \log \Big( 1- \frac{\l N e^N }{\l N e^N +2 } \Big)  \Big\}
	\\ & = \exp \Big\{ -\l t \Big( -1 + \frac{ \g N e^N}{\l N e^N +2} \Big) \Big\} 
	\end{split} \]
since $\log (1-x) \leq -x$. 
Taking $\l = \frac{1}{N e^{N}} $ and recalling that $\g = \frac{6}{N e^{N}} $, it is then easy to check that  
	\[ \exp \Big\{ -\l t \Big( -1 + \frac{ \g N e^N}{\l N e^N +2} \Big) \Big\} 
	 \leq \exp (-\l t) =
	 \exp \Big( - \frac{t}{ N e^{N} } \Big) , \]
which concludes the proof.

\section{Logarithmic fluctuations for IDLA at large times} \label{app:survey}

We survey the proof of the logarithmic fluctuations bound by Jerison, Levine and Sheffield for IDLA on the cylinder graph $\bbZ_N \times \bbZ$, and extend it to larger times to prove Theorem~\ref{th:JLS}. 

\subsection{A priori bound} \label{sec_apriori}
We obtain an a priori bound on the height of $A(T)$ following the outer bound argument by Lawler, Bramson and Griffeath \cite{lawler1992internal}, for arbitrary starting configurations.  
\begin{Lemma} \label{le:a_priori}
Let $(A(t))_{t\geq 0}$ denote an IDLA process on $\bbZ_N \times \bbZ$, and let $h_0 = h(A(0))$ denote its initial height. Then for any $m\geq 3$ there exists $\b =\b (m) \in (0,1)$ such that, for $T\gg N \log N $ and $N$ large enough, it holds 
	\[ \bbP \Big( A(T) \nsubseteq R_{h_0 + \frac{mT}{N}} \Big) \leq \b^{T/N} . \]
Thus if, in particular, the process starts from the flat configuration $A(0)=A_0$, then $h(A(T)) \leq mT/N$ with high probability for $N$ large enough. 
\end{Lemma}
\begin{proof}
Let $Z_k (t) := | A(t) \cap \{ y=k\} |$ denote the number of particles in $A(t)$ at level $k$, and set $\mu_k(t) := \bbE (Z_k(t))$. Then 
	\[ \bbP \Big( A(T) \nsubseteq R_{h_0 + \frac{mT}{N}} \Big)
	 \leq \bbP \Big( Z_{h_0 +\frac{mT}{N}+1}(T) \geq 1 \Big) 
	\leq \mu_{h_0 + \frac{mT}{N}+1} (T) . \]
We claim that 
	\begin{equation} \label{LBG_up}
	 \mu_k(t) \leq \Big( \frac{1}{N} \Big)^{k-h_0-1} \frac{t^{k-h_0}}{(k-h_0)!} 
	 \end{equation}
for all $k>h_0$ and $t\geq 0$. 
Indeed, clearly $\mu_1(t) \leq N$ and $\mu_k(0) =0$  for all $k>h_0$. For other values of $k,j$  we have 
	\[ \begin{split} 
	\mu_k(t+1) - \mu_k(t) & = \bbE ( Z_k (t+1)  - Z_k (t) ) 
	= \bbP ( Y_{t+1} \in A(t)^c \cap \{ y=k\} ) 
	\\ & \leq \frac{1}{N} \bbE ( | A(t) \cap \{ y=k-1\} | ) 
	= \frac{1}{N} \mu_{k-1} (t) , 
	\end{split} \]
where in the above inequality we have used that the probability that the $(t+1)^{th}$ walker reaches level $k-1$ inside $A(t)$ is maximised when $A(t)$ is completely filled up to level $k-2$. Thus for $k>h_0$ 
	\[ \mu_k ( t) = \sum_{s=0}^{t-1} ( \mu_k(s+1) - \mu_k(s) ) \leq 
	\frac{1}{N} \sum_{s=0}^{t-1} \mu_{k-1}(s) , \]
and \eqref{LBG_up} follows by a simple iteration. 
Now take $t=T$, $k =h_0 +\frac{mT}{N}+1$ and recall that $k! \geq k^k e^{-k}$ to get 
	\[ \mu_{h_0 +\frac{mT}{N}+1} (T) 
	\leq \Big( \frac{1}{N} \Big)^{\frac{mT}{N} } 
	\frac{T^{\frac{mT}{N}+1} }{(\frac{mT}{N} +1)!} 
	\leq N \Big[ \Big( \frac{e}{m} \Big)^m \Big]^{\frac{T}{N}}  \leq \b^{T/N} 
	\] 
for any $  \b \in \big(\big(\frac{e}{m}\big)^m , 1)$ and $N$ large enough. 
\end{proof}
\begin{Remark} \label{remark_smallT}
By the above result with $h_0=0$, it suffices to prove Theorem \ref{th:JLS} for large~$T$. Suppose, indeed, that 
there exists a finite constant $b$ such that $T \leq b N \log N$. Since $\frac{T}{N} \leq b \log N$, it suffices to take $a > b$ to have that the inner bound is trivially satisfied. For the outer bound we set $c= \frac{a+b}{3}$ and note that, as long as $a \geq 2b$, it holds 
	\[ \begin{split} 
	\bbP \Big( \big\{ R_{\frac{T}{N} -a\log N}  \subseteq A(T) \subseteq R_{\frac{T}{N} + a \log N } \big\}^c
	\Big) & = \bbP \Big(  A(T) \nsubseteq R_{(a+b) \log N } \Big)
	\\ & \leq \bbP \Big(  A(cN\log N ) \nsubseteq R_{3c \log N } \Big)  
	\\ & \leq \beta ^{c \log N } = N^{-c \log \frac{1}{\b}}, 
	\end{split} \]
which can be made smaller than $N^{-\g}$ by taking $a \geq \max \big\{ 2b ; \frac{3\g}{\log 1/\b } -b \big\}$. In light of this observation, from now on we can assume that $T\gg N \log N$ as $N\to \infty$. 
\end{Remark}

 The proof of Theorem \ref{th:JLS} follows an iterative argument, which we now sketch. 
\begin{Definition}
A point $(x,y)$ with $y\geq 0$ is said to be: 
\begin{itemize}
\item  $m$-early if $(x,y) \in A((y-m)N)$, 
 \item $\ell$-late if $(x,y) \notin A((y+\ell ) N)$.  
 \end{itemize}
\end{Definition}
For $t\leq T$, let 
	\[ \cE_m [t] := \bigcup_{(x,y) \in A(t) } \{ (x,y) \mbox{ is } m\mbox{-early}\} , 
	\qquad 
	\cL_\ell [t] := \bigcup_{(x,y) \in R_{t/N}} \{ (x,y) \mbox{ is } \ell \mbox{-late}\} . 
	\]
Clearly, 
	\[ \big\{ R_{\frac{T}{N} -a\log N}  \subseteq A(T) \subseteq R_{\frac{T}{N} + a \log N }  \big\}^c 
	\subseteq \cE_{a\log N } [T] \cup \cL_{a\log N } [T] ,
	\]
so we bound the probability of the right hand side. Note that the a priori bound in the previous section (cf.\ Lemma \ref{le:a_priori}) with $m=3$ tells us that there exists an absolute constant $\beta \in (0,1)$ such that 
	\[ \bbP \Big( \cE_{\frac{2T}{N}}[T] \Big) \leq \b^{T/N} \ll \b^{\log N}\]
for $N$ large enough.  We use this to initialise the iteration, which consists of showing that, in turn, 
\begin{itemize}
\item no $m$-early point implies no $\ell$-late point ($\ell \ll m $), and 
\item no $\ell$-late point implies no $m'$-early point ($m' \asymp \ell$). 
\end{itemize}
To perform the above steps, we will use an explicit discrete harmonic function with a pole close to the early/late point to build a martingale, which will be then controlled via its quadratic variation.

\subsection{No early points implies no late points}
The main goal of this section is to prove the following. 
\begin{Proposition} \label{pr:early_late}
For any $\g >0$ there exists a finite constant $C=C(\g) $, depending only on $\g$,  such that, if $\ell \geq C \log N$ and $m \leq \ell^2 / C\log N $, then 
	\[ \bbP ( \cL_\ell [T] \cap \cE_m [T]^c ) \leq N^{-(\g +1)} \]
for $N$ large enough. 
\end{Proposition}
The above result tells us that on the event that there is no $m$-early point, there is no $\ell $-late point with high probability, with $\ell = \big\lceil \sqrt{Cm\log N} \big\rceil \ll m $.

To prove this, we argue as follows. For $\z \in \bbZ_N \times \bbZ$, let $L(\z ) = \{ \z \mbox{ is } \ell\mbox{-late}\}$. Then, since  
	\[ \cL_\ell [T] = \bigcup_{\z \in R_{\frac{T}{N}-\ell } } L( \z ) , \]
we have 
	\[ \bbP  ( \cL_\ell [T] \cap \cE_m [T]^c ) \leq \sum_{\z \in R_{\frac{T}{N}-\ell }} 
	\bbP ( L(\z ) \cap \cE_m [T]^c ) . \]
It therefore suffices to show that 
	\[ \bbP ( L(\z ) \cap \cE_m [T]^c ) \leq N^{-(\g +4 )} \] 
for arbitrary $\z \in R_{\frac{T}{N}-\ell}$. 
To see this, we use a discrete harmonic function to build a martingale with pole at $\z$. We then show that on the event $L(\z )$ this martingale is large and negative, while on the event $\cE_m[T]^c $ its quadratic variation is small. This will then imply that the probability that both $L(\z ) $ and $\cE_m[T]^c$ hold is small. 

Recall that  $(A(t))_{t\geq 0}$ denotes the IDLA process.  For $\z = (\z_x , \z_y ) $ and $z \in \bbZ_N \times \bbZ_+$ define
	\[ H_\z (z) := \bbP_z (\mbox{a SRW reaches level }\z_y\mbox{ for the first time at }\z ) , \]
where SRW stands for simple random walk on $\bbZ_N \times \bbZ$. 
Then $H_\z (z) \in [0,1]$ for all $z$ and  $H_\z (z) \to 1/N$ as $z_y \to -\infty$. Moreover,   $H_\z $ is discrete harmonic up to level $\z_y $  (in fact, it is the discrete harmonic extension of the function $\mathbf{1} (x=\z_x )$ at level $\z_y$ to the region $\{ (x,y) : y\leq \z_y \}$). 

We embed the driving random walks in continuous time by mean of time-changed Brownian motions on the lattice (see \cite{jerison2012logarithmic} for a precise definition), that we denote by $\{ (B_n(t))_{t\in [0,1]} , n\geq 1\}$, so that $B_n(0)$ and $B_n(1)$ are the starting and settling location of the $n^{th}$ random walk respectively. This turns out to be technically convenient, since it makes our discrete time martingales into continuous time ones. Finally, 
for real $t\in [0,\infty )$ we define
	\[ M_\z (t) := \sum_{n=1}^{\lfloor t \rfloor } \Big( H_\z ( B_n(1) ) - \frac{1}{N} \Big) 
	+ \Big( H_\z (B_{\lfloor t \rfloor } ( t- \lfloor t \rfloor ) ) - \frac{1}{N} \Big) . \]
Note that the first sum represents the contribution of the settled walkers, while the last term gives the contribution of the active one at time $t$. Since the function $H_\z $ is discrete harmonic up to level $\z_y$, $M_\z$ above is a continuous-time martingale up to the first time the cluster reaches level $\z_y $. 
We split it into a continuous part and a jump part, namely we set $M_\z (t) = M^1_\z  (t) + M^2_\z (t) $ with 
	\[ M^1_\z  (t) = \sum_{n=1}^{\lfloor t \rfloor -1 } \Big( H_\z (B_n(1)) - H_\z (B_n(0)) \Big) 
	+  H_\z (B_{\lfloor t \rfloor } ( t- \lfloor t \rfloor ) ) - B_{\lfloor t \rfloor }(0) , \]
and 
	\[ M^2_\z (t) = \sum_{n=1}^{\lfloor t \rfloor } \Big( H_\z (B_n(0)) - \frac{1}{N} \Big) . \]
Let $S_\z $, $S_\z^1 $ and  $S_\z^2$ denote the quadratic variation of $M_\z$, $M_\z^1$ and $M_\z^2$ respectively. Since  $M_\z^1$ is a continuous martingale starting from $0$, it is  a time-changed Brownian Motion, that is 
	\[ M_\z^1 (t) = B (S_\z^1 (t)) \]
for $B$ one-dimensional Brownian motion. Moreover, since $M_\z^2$ is piecewise constant and by independence of the starting locations, we have 
	\[ S_\z^2 (t) = \sum_{n=1}^{\lfloor t \rfloor } \bbE \Big[ \Big( H_\z (B_n(0)) - \frac{1}{N} \Big)^2\Big] . \]
Recall that we want to bound the quadratic variation of $M_\z$ on the event that no point is $m$-early, which means that the cluster has a controlled height. 
The next lemma shows how a control on the shape of the cluster can be translated into a control of the quadratic variation of the martingale. 

\begin{Lemma} \label{le1}
Assume that $m\gg 1 $ and that $\z_y \geq \frac{t}{N} + 2m +1$. Then 
	\[ \bbE \Big( e^{S_\z (t) }\mathbf{1}_{\cE_{m+1} [t]^c } \Big) \leq e^2 t^{800} . \]
\end{Lemma}
\begin{proof}
We have $S_\z (t) \leq 2( S_\z^1(t) + S_\z^2(t))$, from which 
	\[ \bbE (e^{S_\z (t) }\mathbf{1}_{\cE_{m+1} [t]^c } ) 
	\leq \bbE (e^{2(S_\z^1 (t) + S_\z^2(t) ) \mathbf{1}_{\cE_{m+1} [t]^c } ) } ) 
	\leq \big[ \bbE (e^{4S_\z^1 (t) \mathbf{1}_{\cE_{m+1} [t]^c } } ) 
	\bbE (e^{4S_\z^2 (t) \mathbf{1}_{\cE_{m+1} [t]^c } }) \big]^{1/2}. \]
Let us  estimate the two factors separately. 
As a technical trick, we will assume that the driving random walks are released from a uniformly chosen location at level $- N^2$, so that 
	\begin{equation} \label{trick} 
	 \Big| H_\z (B_n(0)) - \frac{1}{N} \Big| \leq  \frac{1}{N \log N} . 
	 \end{equation} 
This clearly does not change the law of the process, and it has the technical advantage of transferring some amount of the quadratic variation of $M_\z$ from $S_\z^2$ to $S_\z^1$. 
 We have: 
 	\[ \begin{split} 
 	\bbE (e^{4S_\z^2 (t) \mathbf{1}_{\cE_{m+1} [t]^c } }) & \leq 
 	\bbE \Big( \exp \Big\{ 4\sum_{n=1}^{\lfloor t \rfloor } \Big( H_\z(B_n(0)) - \frac{1}{N} \Big)^2 \Big\} \Big) 
 	\\ & \leq \bbE \Big( \exp \Big\{ \frac{4 \lfloor t \rfloor}{(N \log N)^2}  \Big\} \Big) 
 	\leq e^{4} . 
 	\end{split} \]  
Note that the above bound does not depend on $m$, but only on the assumption \eqref{trick}. 
It remains to show that 
	\[ \bbE (e^{4S_\z^1 (t) \mathbf{1}_{\cE_{m+1} [t]^c } } )  \leq t^{1600} . \]
To this end, we use that on the event $\cE_{m+1} [t]^c $ the height on the cluster is controlled, which in turn implies that the quadratic variation of the continuous martingale part cannot be too large. Let us take $t$ to be integer to simplify the writing. Then we find 
	\[ \bbE (e^{4S_\z^1 (t) \mathbf{1}_{\cE_{m+1} [t]^c } } )  
	\leq \bbE \Big( \exp \Big\{ 4\sum_{n=1}^t ( S_\z^1 (n) - S_\z^1 (n-1) ) \mathbf{1}_{\cE_{m+1}[t]^c} \Big\} \Big) . \]
We claim that on the event $\cE_{m+1}[t]^c$ the martingale increments are bounded, that is 
	\[ 	M_\z^1(s) - M_\z^1(\lfloor s \rfloor ) \in [-a , b_n ] \]
for all $s \in [n,n+1) $ and some (non-random) $a,b_n$. Indeed, 
	\[ M_\z^1(s) - M_\z^1(\lfloor s \rfloor )  =
	H_\z ( B_{\lfloor s \rfloor} ( s - \lfloor s \rfloor ) ) - 
	H_\z (B_{\lfloor s \rfloor } (0)) 
	\geq - H_\z (B_{\lfloor s \rfloor } (0))  \geq - \frac{2}{N} =: -a . \]
Moreover, on the event $\cE_{m+1} [t]^c $ we have that for all $n\leq t$ the cluster $A(n)$ is contained in $R_{\frac{n}{N} + m+1}$, from which 
	\[ \begin{split} 
	M_\z^1(s) - M_\z^1(\lfloor s \rfloor ) & =
	H_\z ( B_{\lfloor s \rfloor} ( s - \lfloor s \rfloor ) ) - 
	H_\z (B_{\lfloor s \rfloor } (0))  \\ & 
	\leq \max_{ z : z_y \leq \frac{n}{N} +m+1} ( H_\z (z) - H_\z ( B_n(0)) ) \\ & 
	\leq \frac{10}{\z_y - \frac{n}{N} -m-1 } =: b_n , 
	\end{split} \]
where the last inequality follows from the explicit computation of $H_\z(z)$ for large $N$, and it holds as long as $m\gg 1$ (see \cite{lawler2010random} for details). By a standard argument on Brownian motion,  
this implies that 
\[ S_\z^1(n) - S_\z^1(n-1) \in \t_{n} (-a , b_n) ,  \]
where $\t_n(a,b)$ denotes the exit time from the interval $(-a , b)$ for a one-dimensional Brownian motion starting from $0$. We thus have that 
	\begin{equation} \label{ab} 
	\begin{split} 
	\log \bbE \Big( \exp \Big\{ 4\sum_{n=1}^t ( S_\z^1 (n) - S_\z^1 (n-1) ) \mathbf{1}_{\cE_{m+1}[t]^c} \Big\} \Big) & 
	\leq \log \bbE \Big( \exp \Big\{ 4\sum_{n=1}^t \t_n(-a , b_n) \Big\} \Big)
	\\ & 
	\leq \sum_{n=1}^t \log \bbE \big( e^{4\t_n (-a , b_n)} \big) . 
	\end{split} 
	\end{equation} 
It is easy to check (cf.\ \cite{jerison2012logarithmic}, Lemma 5) that $\bbE (e^{\l \t (-a , b ) }) \leq 1+ 10\l a b \,$ provided $\sqrt{\l} (a+b) \leq 3$. For our choice of parameters $\l =4$ and 
	\[ a+b_n =  \frac{2}{N} + \frac{10}{\z_y - \frac{n}{N} -m-1 } 
	\leq \frac{2}{N} + \frac{10}{\frac{t-n}{N} + m } \leq 1\]
for $N$ large enough, since $m\gg 1$. Thus 
	\[  \begin{split} 
	\sum_{n=1}^t \log \bbE \big( e^{4\t_n (-a , b_n)} \big) 
	& \leq \sum_{n=1}^t \log ( 1+ 40a b_n ) 
	\leq 40 a \sum_{n=1}^t b_n 
	 \leq  400 a \sum_{n=1}^t \frac{1}{\frac{t-n}{N} +1} 
	\\ & \leq \frac{800}{N} \Big( 1 + \int_1^t \frac{dx}{x/N} \Big) 
	= \frac{800}{N} ( 1 + N \log t ) \leq 1600 \log t , 
	\end{split} \]
which concludes the proof.
\end{proof}

We can now show that no early points implies no late points. 
\begin{proof}[Proof of Proposition \ref{pr:early_late}]
Take $m= \left\lfloor \frac{\ell^2}{C\log N } \right\rfloor$ without loss of generality, with $C$ as in the statement, and note that $m\geq \ell \,$ since $\ell \geq C\log N $ by assumption. Let $\z=(\z_x , \z_y )$ be such that $\z_y \leq \frac{T}{N}-\ell$, and set $T_1 = N(\z_y + \ell )$. Note that $T_1 \leq T$. Recall that $L(\z )$ denotes the event that $\z$ is $\ell$-late, that is $\z \notin A(T_1)$. 
As already observed, it will suffice to show that 
	\[ 	 \bbP ( L(\z ) \cap \cE_m [T]^c ) \leq N^{-(\g +4 )} \] 
for arbitrary $\z \in R_{\frac{T}{N}-\ell}$. On the event $L(\z )$ we use the martingale $M_\z$ with pole at $\z$. We will show that both $M_\z (T_1)$ is large and negative, and $S_\z (T_1)$ is small, so that the probability of both happening simultaneously is tiny. 

Let $A_\z (t)$ denote the IDLA cluster at time $t$ with points stopped upon reaching level $\z_y$, counted with multiplicity. We have 
	\[ M_\z (T_1) = \sum_{A_\z (T_1) } \Big( H_\z (z) - \frac{1}{N} \Big) . \]
On the event $L(\z )$, no particle reaches $\z$ by time $T_1$. Moreover, $H_\z (z) =0$ for all $z \neq \z $ at level $\z_y$, while $H_\z (z)>0$ if $z_y < \z_y$. It follows that the sum defining $M_\z (T_1)$ is maximised when as many particles as possible are below level $\z_y$, that is when the cluster is completely filled up to level $\z_y -1$. Hence 
	\[ \begin{split} 
	M_\z (T_1) & \leq \underbrace{\sum_{z \in R_{\z_y -1}} \Big( H_\z (z) - \frac{1}{N} \Big) }_{0 \mbox{ (mean value pr.)}}
	+ \sum_{z \in A_\z(T_1) , \, z_y = \z_y } \Big( H_\z (z) - \frac{1}{N} \Big) 
	\\ & = -\frac{1}{N} \underbrace{\sharp \{ \mbox{particles stopped at level }\z_y \} }_{T_1 - N(\z_y -1 ) = N(\ell +1 ) } 
	\leq -\ell . 
	\end{split} \]
To show that, on the other hand, the quadratic variation is small, we use the following lemma. 
\begin{Lemma} \label{le2}
Assume that $m\gg 1$ and $\ell \leq m$. Fix any $\z$, and let $t = N(\z_y + \ell ) $. Then 
	\[ \bbE ( e^{S_\z (t) } \mathbf{1}_{\cE_m[t]^c} ) \leq e^{200m} t^{800} . \]
\end{Lemma}
Using the above with $t=T_1$ and since $T_1 \leq e^m$ we gather that
	\[ \bbE ( e^{S_\z (T_1) } \mathbf{1}_{\cE_m[T_1]^c} ) \leq e^{200m} T_1^{800} \leq e^{1000m} . \]
In conclusion, for any $s>0$ we have 
	\[ \bbP (L(\z ) \cap \cE_m[T]^c ) \leq 
	\bbP ( \cE_m[T]^c \cap \{ S_\z (T_1) > s \} ) + 
	\bbP ( \{ S_\z (T_1) \leq  s \} \cap L(\z )  ) , \]
with 
	\[  \bbP ( \cE_m[T]^c \cap \{ S_\z (T_1) > s \} ) 
	\leq \bbP (S_\z (T_1) \mathbf{1}_{\cE_m[T_1]^c} > s ) 
	\leq \bbE ( e^{S_\z (T_1) } \mathbf{1}_{\cE_m[T_1]^c} ) e^{-s} 
	\leq e^{-1000m} \]
taking $s=2000m$, and 
	\[ \bbP ( \{ S_\z (T_1) \leq  s \} \cap L(\z )  ) 
	\leq \bbP (  S_\z (T_1) \leq  s , \, M_\z (T_1) \leq - \ell   ) \leq e^{-\ell^2 / 2s } \]
by a standard large deviations estimate for Brownian motion. In all, 
	\[ \bbP ( L(\z ) \cap \cE_m[T]^c ) \leq 
	e^{-1000m } + e^{- \ell^2 / 4000m} \leq 2 N^{-(\g+5)} , \]
as long as $m\geq \ell \geq (\g +5 ) \log N $ and $\frac{\ell^2}{4000m} \geq (\g +5 ) \log N $. This shows that Proposition~\ref{pr:early_late} holds with $C \geq 4000 (\g +5 )$. 

In order to conclude, it remains to prove the above lemma. 
\begin{proof}[Proof of Lemma \ref{le2}]
As in the proof of Lemma \ref{le1}, we have 
	\[  \bbE (e^{S_\z (t) }\mathbf{1}_{\cE_{m} [t]^c } ) 
	\leq \bbE (e^{2(S_\z^1 (t) + S_\z^2(t) ) \mathbf{1}_{\cE_{m} [t]^c } ) } ) 
	\leq \big[ \bbE (e^{4S_\z^1 (t) \mathbf{1}_{\cE_{m} [t]^c } } ) 
	\bbE (e^{4S_\z^2 (t) \mathbf{1}_{\cE_{m} [t]^c } }) \big]^{1/2}. \]
The factor involving $S_\z^2$ is bounded above by an absolute constant, so it suffices to show that, say, 
	\[ 	\bbE (e^{4S_\z^1 (t) \mathbf{1}_{\cE_{m+1} [t]^c } } ) 
	\leq e^{320m} t^{1600} . \]
	Note that $\z_y = \frac{t}{N} -\ell$, so we cannot use Lemma \ref{le1}, which requires $\z_y \geq \frac{t}{N} + 2m+1$. On the other hand, note that if $t_0 \leq t$ then $\cE_m[t]^c \subseteq \cE_m[t_0]^c$. Moreover, for $t_0$ such that $\z_y \geq \frac{t_0}{N} + 2m+1$ we can apply Lemma \ref{le1}. Take therefore $t_0 = N(\z_y -2m-1 ) \vee 0 $. Then 
		\[ \begin{split} 
			\bbE (e^{4S_\z^1 (t) \mathbf{1}_{\cE_{m+1} [t]^c } } ) & 
			\leq \Big[ \bbE (e^{8 ( S_\z^1 (t) - S_\z^1(t_0))  \mathbf{1}_{\cE_{m+1} [t]^c } } ) 
				\underbrace{\bbE (e^{8S_\z^1 (t_0) \mathbf{1}_{\cE_{m+1} [t]^c } } ) }_{\leq t^{3200}} \Big]^{1/2} 
			\\ & \leq \Big[ \bbE \Big( \exp \Big\{ 8\sum_{n=t_0+1}^t ( S_\z^1(n) - S_\z^1(n-1)) \mathbf{1}_{\cE_m[t_0]^c } \Big\} \Big) \Big]^{1/2}  t^{1600} 
			\\ & \leq \Big[ \underbrace{\bbE \big( e^{8 \t ( -\frac{2}{N} , 1 ) } \big)}_{\leq 1 + \frac{160}{N} \leq e^{160/N} } \Big]^{\frac{t-t_0}{2}} t^{1600} 
			\leq e^{\frac{80}{N}(t-t_0) } t^{1600} . 
			\end{split} \]
Now, if $t_0=0$ then $\z_y \leq 2m +1$, so $\frac{t}{N} = \z_y + \ell \leq 2m +1+\ell \leq 4m $. Otherwise, $\frac{t-t_0}{N} = \ell + 2m + 2 \leq 4m$. In all, 
	\[ \bbE (e^{4S_\z^1 (t) \mathbf{1}_{\cE_{m+1} [t]^c } } ) 
	\leq e^{320m} t^{1600}, \]
as wanted. 
\end{proof} 
This concludes the proof of Proposition \ref{pr:early_late}. 
\end{proof}

\subsection{No late points implies no early points} 
We prove the following. 
\begin{Proposition} \label{pr:late_early} 
Assume $m,\ell \leq \sqrt{3C N \log N }$, with $C$ as in Proposition \ref{pr:early_late}. Then 
for any $\g >0$, there exist a finite constant $C'=C'(\g) $, depending only on $\g$,  and an absolute constant $b>0$ such that, if $m \geq C' \log N $ and $\ell \leq  \frac{bm}{48\pi} $, it holds 
	\[ \bbP ( \cE_m [T] \cap \cL_\ell [T]^c ) \leq N^{-(\g +1)} \]
for $N$ large enough. 
\end{Proposition}
\begin{proof}
We assume $\ell = \frac{bm}{48\pi} $ without loss of generality, with $b$ to be determined. 
Let 
	\[ Q_{z,t } := \{ z \mbox{ is the first }m\mbox{-early point, absorbed at time }t\} . \]
Then 
	\[ \bbP ( \cE_m [T] \cap \cL_\ell [T]^c )  = \sum_{t=1}^{T} \sum_{z\in R_{T}} 
	\bbP ( Q_{z,t} \cap \cL_\ell [T]^c ) . \]
On the event $Q_{z,t} $ take $\z = (\z_x , \z_y )$ such that $\z_y = z_x$ and $\z_y = z_y + m+1 = \frac{t}{N} + 2m +1 $. Then for any $s>0$ 
	\[ \begin{split} 
	\bbP ( Q_{z,t} \cap \cL_\ell [T]^c )  \leq & 
	\bbP ( Q_{z,t} \cap \{ S_\z (t) > s \} ) + \\ & 
	 \bbP \Big( Q_{z,t} \cap \big\{ M_\z (t) < \frac{b}{1000} m \big\} \cap  \cL_\ell [T]^c \Big) + \\
	&   \bbP \Big( \{ S_\z (t) > s \} \cap \big\{ M_\z (t) \geq  \frac{b}{1000} m \big\} \Big)  .
	\end{split} \]
Taking $s = (2\g + 1000 ) \log N$, we find 
	\[ \bbP \Big( \{ S_\z (t) > s \} \cap \big\{ M_\z (t) \geq  \frac{b}{1000} m \big\} \Big) 
	\leq e^{-s/2} \leq N^{- (\g +500 )} , \]
and, using that $Q_{z,t} \subseteq \cE_{m+1}[t]^c$, 
	\[ \begin{split} 
	\bbP ( Q_{z,t} \cap \{ S_\z (t) > s \} ) & \leq 
	\bbP ( \cE_{m+1}[t]^c \cap \{ S_\z (t) > s \} )  
	= \bbP ( S_\z (t) \mathbf{1}_{\cE_{m+1}[t]^c} > s ) 
	\\ & \leq \bbE ( e^{S_\z (t) } \mathbf{1}_{\cE_{m+1}[t]^c} ) e^{-s} 
	\leq e^8 t^{800 } e^{-s } \leq N^{-(2\g + 100) }. 
	\end{split} \]
To finish, we show that 
	\[ \bbP \Big( Q_{z,t} \cap \big\{ M_\z (t) < \frac{b}{1000} m \big\} \cap  \cL_\ell [T]^c \Big) \leq N^{-(\g +5 )} . \]
Note that on the event $Q_{z,t}$ we have $A(t) \subseteq R_{\frac{t}{N} + m+1} $. Let $B_r(z)$ denote the Euclidean ball of radius $r$ around $z$. Partition $A(t)$ as follows: 
	\[ A(t) = \Big( \underbrace{ A(t) \cap R_{\frac{t}{N} -\ell } }_{A_1} \Big) \cup 
	\Big( \underbrace{ A(t) \cap B_m(z) }_{A_2} \Big) \cup \Big( \underbrace{A(t) \setminus ( A_1 \cup A_2 ) }_{A_3} \Big) . \]
Then 
	\[ M_\z (t) = \sum_{z\in A_1 } \Big( H_\z (z) -\frac{1}{N} \Big) + 
	\sum_{z\in A_2 } \Big( H_\z (z) -\frac{1}{N} \Big) + 
	\sum_{z\in A_3 } \Big( H_\z (z) -\frac{1}{N} \Big) . \]
Since $A_1 = R_{\frac{t}{N} - \ell } $ on the event $\cL_\ell [T]^c $, 
	\[ \sum_{z\in A_1 } \Big( H_\z (z) -\frac{1}{N} \Big) =0 . \]
Moreover, 
	\[ \sum_{z\in A_3 } \Big( H_\z (z) -\frac{1}{N} \Big) \geq -\frac{|A_3|}{N} = 
	- \frac{t-(t-\ell N ) - |A_2|}{N}  \geq - \ell . \]
It remains to estimate the contribution to the martingale of points in $A_2$, which we show to be large. To this end, observe that if $i\leq m$ and $m\leq k \leq 2m+1$, then (with $m\gg 1$) 
	\[ \begin{split}
	\bbP_0 ( \mbox{a SRW reaches level }k\mbox{ for the first time at }(i,k) ) 
	& \geq \frac{1}{4\pi} \log \Big( 1+\frac{4k}{i^2 + (k-1)^2} \Big) + \cO \Big( 
	\frac{1}{N^2}\Big) \\ & 
	\geq \frac{1}{2\pi} \frac{k}{i^2 + (k-1)^2 } \geq \frac{1}{12\pi m } . 
	\end{split} \]
This implies that $H_\z (z) \geq \frac{1}{12\pi m }$  for $z\in A_2$. Moreover, the next lemma shows that $A_2$ contains a positive proportion of points, and it identifies the absolute constant $b$ in Proposition \ref{pr:late_early}. 
\begin{Lemma} [Thin tentacles, cf.\ \cite{jerison2012logarithmic} Lemma 2]
Assume $m\ll N$. Then there exist absolute constants $b, C_0 , c_0$ such that for all $z $ with $z_y \geq m$ it holds 
	\[ \bbP ( z \in A(t) , \; |A(t) \cap B_m(z)| \leq bm^2 ) \leq C_0 e^{-c_0 m} , \]
for all $t$. 
\end{Lemma}
\begin{proof}
This can be proven exactly as in the $\bbZ^2$ case treated in \cite{jerison2012logarithmic}, as the argument is completely local and $m \ll N$. 
\end{proof}
This allows us to conclude that, on the event $Q_{z,t} \cap \cL_\ell [T]^c \cap \{ |A(t) \cap B_m(z)| \leq bm^2 \}$,  it holds 
	\[ M_\z (t) \geq \sum_{z\in A_2 } \Big( H_\z (z) - \frac{1}{N} \Big) - \ell 
	\geq \Big( \frac{1}{12\pi m} - \frac{1}{N} \Big) |A_2| - \ell 
	\geq \frac{bm}{12\pi } - \frac{bm^2}{N} - \ell 
	\geq \frac{bm}{48\pi}   , \]
where the last inequality follows by recalling that $\ell = b m/48\pi $ and that, by assumption, $m^2 \leq 3C N\log N$ and $m\geq C'\log N$, from which 
	\[ \frac{bm^2}{N} \leq 3bC \log N  \leq \frac{bm}{24\pi} \]
as long as $3C \leq C' / 24\pi$, which can always be ensured by taking $C'$ large enough, depending only on $C$. 

In all, we conclude that 
	\[ \begin{split} 
	 \bbP \Big( Q_{z,t} \cap \big\{ M_\z (t) < & \frac{b}{1000 } m \big\}  \cap  \cL_\ell [T]^c \Big)    = \\ & = 
	\bbP \Big( Q_{z,t} \cap \big\{ M_\z (t) < \frac{b}{1000 } m \big\} \cap  \cL_\ell [T]^c 
	\cap \{ |A(t) \cap B_m(z)| \leq bm^2 \} \Big) \\ & 
	\leq \bbP ( z\in A(t) , \{ |A(t) \cap B_m(z)| \leq bm^2 \} ) \\ & \leq C_0 e^{-c_0 m }  
	\leq N^{-(\g +5 )} , 
	\end{split} \]
where the last inequality holds as long as $m \geq \frac{2}{c_0} (\g +5 ) \log N$, which can always be ensured at cost of increasing $C'$, only depending on $\g$. This concludes the proof of Proposition~\ref{pr:late_early}. 
\end{proof}

\subsection{Iterative scheme}
To finally prove Theorem \ref{th:JLS} we iteratively apply Propositions \ref{pr:early_late} and \ref{pr:late_early}, starting with $m=2T/N \gg \log N $ and iterating as long as $m,\ell \geq a_\g \log N$, with $a_\g =\max \{ C ,  C' \}$. 
Assume that $T\gg N \log N$, otherwise the result holds by Remark \ref{remark_smallT}. 
We have from Lemma \ref{le:a_priori} with $m=3$ that 
	\[ \bbP (A(T) \nsubseteq R_{3T/N} ) \leq  \bbP (A(N^2 \log^2 N ) \nsubseteq R_{3T/N} )
	\leq \b^{N} \leq N^{\g +1} ,\]
where, for all fixed $\g >0$,  the last inequality holds for $N$ large enough. 
Let $m_0 = 3N \log^2 N $. For $C,C'$ as in Proposition \ref{pr:early_late} and \ref{pr:late_early} respectively, iteratively define 
	\[ \begin{cases} 
	\ell_k = \sqrt{ C m_{k-1} \log N } , \\
	m_k = \frac{48 \pi }{b} \ell_k , 
	\end{cases} 
	\quad k\geq 1. \]
Then, as long as $m_k , \ell_k \geq a_\g \log N$, we have 
	\[ \begin{split} 
	& \bbP ( \cE_{m_{k-1}} [T]^c \cap \cL_{\ell_k} [T]^c ) \leq N^{-(\g +1)}  , \\
	& \bbP (  \cL_{\ell_k} [T]^c \cap \cE_{m_k} [T]  ) \leq N^{-(\g +1)} 
	\end{split} \]
by Proposition \ref{pr:early_late} and \ref{pr:late_early} respectively. 
Thus 
	\[ \begin{split} 
	 & \bbP ( \cE_{m_0} [T] ) \leq N^{-(\g+1)} , \\
	 & \bbP ( \cL_{\ell_1} [T] ) \leq \bbP ( \cL_{\ell_1} [T] \cap \cE_{m_0}[T]^c ) 
	 + \bbP ( \cE_{m_0} [T] ) \leq 2N^{-(\g+1)} , \\
	 & \bbP ( \cE_{\m_1} [T] ) \leq \bbP ( \cE_{m_1} [T] \cap \cL_{\ell_1}[T]^c ) 
	 + \bbP ( \cL_{\ell_1} [T] ) \leq 3N^{-(\g+1)} , \\
	 & \quad \vdots \\
	 & \bbP ( \cL_{\ell_k}[T]) \leq 2k N^{-(\g+1)} , \\
	 & \bbP ( \cE_{m_k}[T]) \leq (2k+1)N^{-(\g +1)} . 
	\end{split} \]
We want to iterate as much as possible, that is to take $k$ as large as possible so that $\ell_k , m_k \geq a_\g \log N$. 
It is easy to check that if $k > c_0 \log \log N$, for a suitable absolute constant $c_0$, then $\ell_k < a_\g \log N$, so we can iterate at most $c_0 \log \log N$ times. In all, this shows that by taking $\ell  = \max \{ \ell_{c_0 \log \log N} ; a_\g \log N \} $ and $m = \max \{ m_{c_0 \log \log N} ; a_\g \log N \}$ we have 
	\[ \bbP ( \cE_m[T] \cup \cL_\ell [T] ) \leq \bbP ( \cE_m[T] ) + \bbP ( \cL_\ell [T] ) 
	\leq 5 c_0  N^{-(\g+1)} \log \log N \leq N^{-\g} , \]
for $N$ large enough, as wanted.

\medskip

\bibliography{HLbib2}
\bibliographystyle{plain}

\end{document}